\documentclass[a4paper,9pts]{article}

\usepackage{cite}
\usepackage{amsmath,amsthm,amssymb}
\usepackage{fixltx2e}
\usepackage{url}
\usepackage{array}
\usepackage{subcaption}
\usepackage{pgfplots}
\pgfplotsset{compat=newest}
\usetikzlibrary{plotmarks}
\usepackage{grffile}

\usetikzlibrary{spy,calc}\usepackage{pgfplots}
\usepackage{graphicx}
\usepackage{dsfont}

\usepackage{algorithm,algpseudocode}

\newtheorem{remark}{Remark}

\newtheorem{proposition}{Proposition}

\def\R{\mathbb{R}}
\def\N{\mathbb{N}}
\def\prox{\mathrm{prox}}
\def\lse{\mathrm{lse}}

\def\P{\mathrm{p}}

\def\argmin{\mathop{\mathrm{argmin}}}

\graphicspath{{./images/}}

\begin{document}

\title{Multiview Attenuation Estimation and Correction}

\author{Valentin~Debarnot\thanks{V. Debarnot is a student at INSA, Universit\'e de Toulouse, France}, Jonas~Kahn and Pierre~Weiss
\thanks{J. Kahn \& P. Weiss are with IMT and ITAV, CNRS and Universit\'e de Toulouse, France}
}

\markboth{Submitted}%
{Absorption from multi-view imaging}

\maketitle

\begin{abstract}
Measuring attenuation coefficients is a fundamental problem that can be solved with diverse techniques such as X-ray or optical tomography and lidar. 
We propose a novel approach based on the observation of a sample from a few different angles. 
This principle can be used in existing devices such as lidar or various types of fluorescence microscopes. 
It is based on the resolution of a nonlinear inverse problem. 
We propose a specific computational approach to solve it and show the well-foundedness of the approach on simulated data. 
Some of the tools developed are of independent interest. 
In particular we propose an efficient method to correct attenuation defects, new robust solvers for the lidar equation as well as new efficient algorithms to compute the Lambert W function and the proximal operator of the logsumexp function in dimension 2.
\end{abstract}



\section{Introduction}

The ability to analyze the composition of gases in the atmosphere, the organization of a biological tissue, or the state of organs in the human body has invaluable scientific and societal repercussions. 
These seemingly unrelated issues can be solved thanks to a common principle: rays traveling through the sample are attenuated and this attenuation provides an indirect measurement of absorption coefficients.
This is the basis of various devices such as X-ray and optical projection tomography \cite{natterer1986mathematics,sharpe2002optical,vermeer2014depth} or lidar \cite{weitkamp2006lidar}. 
The aim of this paper is to provide an alternative approach based on the observation of the sample from a few different angles.

\subsection{The basic principle}\label{subsec:principle}

Let us provide a flavor of the proposed idea in an idealized 1D system. 
Assume that two measured signals $u_1$ and $u_2$ are formed according to the following model:
\begin{equation}\label{eq:eq1}
 u_1(x) = \beta(x) \exp\left(-\int_{0}^x \alpha(t)\,dt\right) \mbox{ for } x\in [0,1]
\end{equation}
and
\begin{equation}\label{eq:eq2}
 u_2(x) = \beta(x) \exp\left(-\int_{x}^1 \alpha(t)\,dt\right) \mbox{ for } x\in [0,1].
\end{equation}
The function $\beta:[0,1]\to \R_+$ will be referred to as a density throughout the paper. 
It may represent different physical quantities such as backscatter coefficients in lidar or fluorophore densities in microscopy. 
The function $\alpha:[0,1]\to \R_+$ will be referred to as the attenuation and may represent absorption or extinction coefficients. 
The signals $u_1$ and $u_2$ can be interpreted as measurements of the same scene under opposite directions.
Equations \eqref{eq:eq1} and \eqref{eq:eq2} coincide with the Beer-Lambert law that is a simple model to describe attenuation of light in absorbing media. 
The question tackled in this paper is: \emph{can we recover both $\alpha$ and $\beta$ from the knowledge of $u_1$ and $u_2$?}

Under a positivity assumption $\beta(x)>0$ for all $x\in [0,1]$, the answer is straightforwardly positive. 
Setting $v(x)=\log\left(\frac{u_2(x)}{u_1(x)}\right)$, equations \eqref{eq:eq1} and \eqref{eq:eq2} yield:
\begin{equation}
 v(x) = \int_{0}^x \alpha(t)\,dt -\int_{x}^1 \alpha(t)\,dt.
\end{equation}
Therefore
\begin{equation}\label{eq:directinversionformula}
 \alpha(x)= \frac{1}{2}\frac{\partial}{\partial_x}v(x)
\end{equation}
and
\begin{equation}\label{eq:directinversionformula2}
 \beta(x) = \frac{u_1(x)}{\exp\left(-\int_{0}^x \alpha(t)\,dt\right)}.
\end{equation}
Unfortunately, formulas \eqref{eq:directinversionformula} and \eqref{eq:directinversionformula2} only have a theoretical interest: they cannot be used in practice since computing the derivative of a log of a ratio is extremely unstable from a numerical point of view.
We will therefore design a numerical procedure based on a Bayesian estimator to retrieve the density $\alpha$ and attenuation coefficient $\beta$ in a stable and efficient manner. It is particularly relevant when the data suffer from Poisson noise. 

\subsection{Contributions}

This paper contains various contributions listed below.
\begin{itemize}
 \item We show that it is possible to retrieve attenuation coefficients from multiview measurements in different systems such as lidar, confocal or SPIM microscopes. To the best of our knowledge, this fact was only known in lidar until now  \cite{kunz1987bipath,hughes1988double,cuesta2010lidar}\footnote{We became aware of these works while finishing this work.}.
 
 Fig. \ref{fig:illustrationofprinciple} summarizes the proposed idea. The attenuation, which is usually considered as a nuisance in confocal microscopy is exploited to measure attenuation. The algorithm successfully retrieves estimates of the density and attenuation from two attenuated and noisy images.
Let us mention that some researchers already proposed to measure absorption and correct attenuation by combining optical projection tomography and SPIM imaging \cite{mayer2014optispim}. The principle outlined here shows that much simpler optical setups (a traditional confocal microscope) theoretically allows estimating the same quantities. 
 \item We propose novel Bayesian estimators for the density $\alpha$ and the attenuation $\beta$. To the best of our knowledge, this is the first attempt to provide a clear statistical framework to recover these quantities. 
 \item The proposed estimators are solutions of a nonconvex problem. We develop a specific nonlinear programming solver based on a warm-start convex initialization. This leads us to develop an efficient algorithm to compute the proximal operator of the logsumexp function in dimension 2 as well as a new algorithm to compute the Lambert W function.
 \item The proposed estimators also seem to be novel for the standard mono-view inverse problem in lidar and for correcting attenuation defects under a Poisson noise assumption.
 \item We perform a numerical validation of the proposed ideas on synthetic data, showing the well-foundedness of the approach. The validation of the method on specific devices is left as an outlook for future works.
\end{itemize}

\begin{figure}
\captionsetup[subfigure]{justification=centering}
 \centering
  \begin{subfigure}[b]{0.35\textwidth}
    \includegraphics[width=\textwidth]{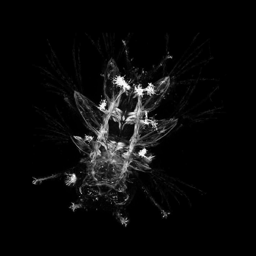}
    \caption{Density $\beta$}\label{fig:density}
  \end{subfigure}
  \begin{subfigure}[b]{0.35\textwidth}
    \includegraphics[width=\textwidth]{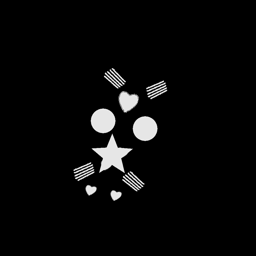}
    \caption{Attenuation $\alpha$}\label{fig:attenuation}
  \end{subfigure}
  \begin{subfigure}[b]{0.35\textwidth}
    \includegraphics[width=\textwidth]{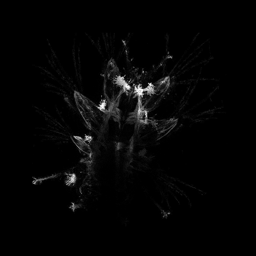}
    \caption{Image $u_1$}\label{fig:u1}
  \end{subfigure}
  \begin{subfigure}[b]{0.35\textwidth}
    \includegraphics[width=\textwidth]{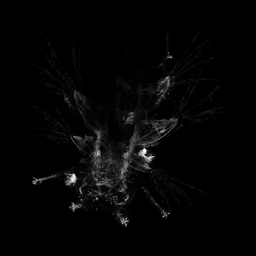}
    \caption{Image $u_2$}\label{fig:u2}
  \end{subfigure}
  \begin{subfigure}[b]{0.35\textwidth}
    \includegraphics[width=\textwidth]{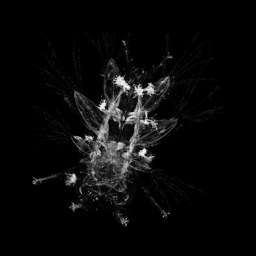}
    \caption{Estimated density\\ SNR=22.4dB}\label{fig:alpha}
  \end{subfigure}
  \begin{subfigure}[b]{0.35\textwidth}
    \includegraphics[width=\textwidth]{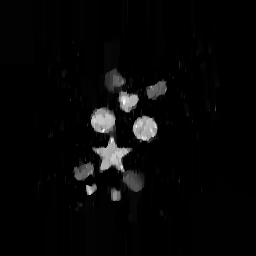}
    \caption{Estimated attenuation\\ SNR=9.4dB}\label{fig:beta}
  \end{subfigure}  
  \caption{Illustration of the contribution. A sample (here an insect) has a fluorophore density $\alpha$ shown in Fig. \ref{fig:density} and an attenuation map $\beta$ shown in Fig.\ref{fig:attenuation}.
The two measured images $u_1$ and $u_2$ are displayed in Fig. \ref{fig:u1} and \ref{fig:u2}.
As can be seen, they are attenuated differently (top to bottom and bottom to top) since the optical path is reversed. 
From these two images, our algorithm provides a reliable estimate of each map in Fig. \ref{fig:alpha} and \ref{fig:beta} despite Poisson noise.}\label{fig:illustrationofprinciple}
\end{figure}

\section{Applications}\label{sec:applications}

In this section, we show various applications where the methodology proposed in this paper can be applied.
We finish by describing precisely the discrete model considered in our numerical experiments.

\subsection{Lidar}

In lidar, an object (atmosphere, gas,...) is illuminated with a laser beam. 
Particles within the object reflect light. 
The time to return of the reflected light is then measured with a scanner. 
The received signal $u_1(x)$  is the backscattered mean power at altitude $x$ for a specific wavelength.
The density $\beta$ corresponds to the backscattered coefficient, while $\alpha$ is called extinction coefficient. The equation relating $u_1$ to $\alpha$ and $\beta$ is: 
\begin{equation}\label{eq:lidareq}
 u_1(x) = \mathcal{P}\left( \frac{C}{x^2} \beta(x) \exp\left( -2 \int_{0}^x \alpha(t)\, dt\right) \right),
\end{equation}
where $C$ is independent of $x$. 
The notation $\mathcal{P}(z)$ stands for a Poisson distributed random variable of parameter $z$. 
The Poisson distribution is a rather good noise model in lidar, since measurements describe a number of detected photons. 
The term $\frac{C}{x^2} \beta(x)$ appears in the lidar equation \eqref{eq:lidareq} instead of simply $\beta$.
The algorithm developed later will allow retrieving $\frac{C}{x^2} \beta(x)$ instead of $\beta$.
This is not a problem since there is a direct known relationship between both. 
\begin{remark}
In Raman lidar, the coefficient $\beta$ corresponds to the  molecular density of the atmosphere, while $\alpha$ is the sum of extinction coefficients at different wavelengths. The theory developed herein also applied to this setting.
\end{remark}

When the backscatter coefficient $\beta$ has a known analytical relationship with the extinction coefficient $\alpha$, direct inversion is possible. A popular method is Klett's formula \cite{klett1981stable} for instance. Alternative formula exist \cite{ansmann1990measurement} when the backscatter coefficient is known. 
The recent trend consists in using iterative methods coming from the field of inverse problems \cite{shcherbakov2007regularized,pornsawad2008ill,garbarino2016expectation}, leading to improved robustness. 
All these approaches crucially depend on a precise knowledge of the backscatter coefficient. 
This is a strong hypothesis that is often rough or unreasonable in practice. 

To overcome this issue, a few authors proposed to use two opposite lidars and to retrieve the attenuation coefficient using equation \eqref{eq:directinversionformula} \cite{kunz1987bipath,hughes1988double,cuesta2010lidar}. 
The stability to noise was ensured by linear filtering of the input and output data.
The algorithms proposed later will provide a more robust and statistically motivated approach.

\subsection{Confocal microscopy}

The principle proposed herein can also be applied to fluorescence microscopy, in particular confocal and multi-SPIM microscopy. To the best of our knowledge, this idea is novel.
We expect it to be relevant when absorption dominates scattering and diffusion. 
A particular case of interest should be optically cleared samples that are commonly used in optical projection tomography \cite{sharpe2002optical}. 

In confocal microscopy, a laser beam is focused at a specific point in 3D space and excites fluorophores. The emitted light is then collected on a camera. A 3D image can be created by scanning the whole sample volume with the focal spot, see Fig. \ref{fig:confocal}.

\begin{figure}
 \centering
 \includegraphics[width=0.35\textwidth]{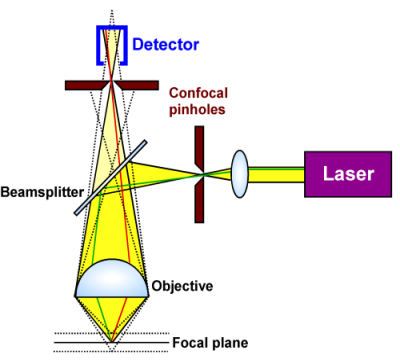}
 \caption{Scheme of a confocal microscope}\label{fig:confocal}
\end{figure}

In order to apply the proposed principle, two images from opposite sides have to be taken. 
This can be done by either rotating the sample or using $4\pi$ microscopes \cite{cremer1974considerations,hell1992properties}. 
The image formation model can then be written as follows:
\begin{equation}\label{eq:image1}
 u_1(x,y,z) = \mathcal{P}\left( C\beta(x,y,z) \exp\left( - (A^+_t \alpha_t+A^+_e \alpha_e)(x,y,z)\right) \right)
\end{equation}
and
\begin{equation}\label{eq:image2}
  u_2(x,y,z) = \mathcal{P}\left( C\beta(x,y,z) \exp\left( - (A^-_t \alpha_t+A^-_e \alpha_e)(x,y,z) \right)\right).
\end{equation}
The coefficients $\alpha_t$ and $\alpha_e$ refer to the attenuation coefficients for the transmitted and emitted light respectively. They may differ since the excitation light is typically red, while the emitted light is usually green. 
In order to retrieve them, we will assume a linear relationship of type $\alpha = \alpha_e = c \alpha_t$ between both for a certain constant $c$. 
The symbols $A_t^+$, $A_e^+$, $A_t^-$ and $A_t^+$ are integral operators describing the optical path for the transmitted and emitted light respectively. In the numerical experiments of this paper, we will simply use the following model:
\begin{equation}
 (A^t_1 \alpha_t)(x,y,z)= \int_{0}^z \alpha_t(x,y,t)\,dt
\end{equation}
\begin{equation}
 (A^t_2 \alpha_t)(x,y,z)= \int_{1}^z \alpha_t(x,y,t)\,dt
\end{equation}
\begin{equation}
 (A^e_1 \alpha_e)(x,y,z)= \int_{0}^z \alpha_e(x,y,t)\,dt
\end{equation}
\begin{equation}
 (A^e_2 \alpha_e)(x,y,z)= \int_{1}^z \alpha_e(x,y,t)\,dt.
\end{equation}
The proposed algorithm also applies to more general linear operators integrating the attenuation coefficients along cones. 
With the mentioned hypotheses and setting $C=1$, equations \eqref{eq:image1} and \eqref{eq:image2} simplify to:
\begin{equation}\label{eq:image1_}
 u_1(x,y,z) = \mathcal{P}\left(\beta(x,y,z) \exp\left( - (A_1 \alpha)(x,y,z)\right) \right)
\end{equation}
and
\begin{equation}\label{eq:image2_}
 u_2(x,y,z) = \mathcal{P}\left(\beta(x,y,z) \exp\left( - (A_2\alpha)(x,y,z)\right)\right),
\end{equation}
where $A_1 = A^t_1 + c A^e_1$ and $A^2 = A^t_2 + c A^e_2$. 

\subsection{Multiview SPIM}

SPIM is an acronym for Selective Plane Illumination Microscopy \cite{huisken2004optical}.
An alternative name is Light Sheet Fluorescence Microscopy (LSFM). 
Its principle is explained in Fig. \ref{fig:SPIM}, left: a laser beam passing through a cylindrical lens focuses in only one direction, getting the shape of a light sheet. This light sheet excites fluorophores in a sample. The light sheet coincides with the focal plane of a microscope, allowing imaging 2D sections of the object. A 3D image can be reconstructed by shifting the sample or the light sheet and stacking the obtained images.

Many different variants of this microscope exist. Here, we are interested in the multiview SPIM \cite{huisken2007even,krzic2012multiview,tomer2012quantitative,chhetri2015whole}, which simultaneously produces 4 images generated with different optical paths, see Fig. \ref{fig:SPIM}, right and Fig. \ref{fig:schememspim}. Similarly to the confocal microscope, let $\alpha_t$ and $\alpha_e$ denote the attenuation maps for the transmitted and emitted light respectively. The four images can be written as:
\begin{equation}
 u_{1,+} = \mathcal{P}\left( C\beta \exp(-A_{(1,t,+)}\alpha_t - A_{(1,e,+)}\alpha_e)\right)
\end{equation}
\begin{equation}
 u_{1,-} = \mathcal{P}\left( C\beta \exp(-A_{(1,t,-)}\alpha_t - A_{(1,e,-)}\alpha_e)\right)
\end{equation}
\begin{equation}
 u_{2,+} = \mathcal{P}\left( C\beta \exp(-A_{(2,t,+)}\alpha_t - A_{(2,e,+)}\alpha_e)\right)
\end{equation}
\begin{equation}
 u_{2,-} = \mathcal{P}\left( C\beta \exp(-A_{(2,t,-)}\alpha_t - A_{(2,e,-)}\alpha_e)\right).
\end{equation}
The eight operators $A_{\cdot,\cdot,\cdot}$ describe integrals along the different directions for the emitted and transmitted waves. It is relatively easy to show, using a reasoning similar to that of paragraph \ref{subsec:principle}, that the four images allow retrieving the attenuation coefficients $\alpha_e$ and $\alpha_t$, without assuming linear relationships between them. This is an advantage over the simpler confocal microscope.

\begin{figure}
 \centering
 \includegraphics[width=0.35\textwidth]{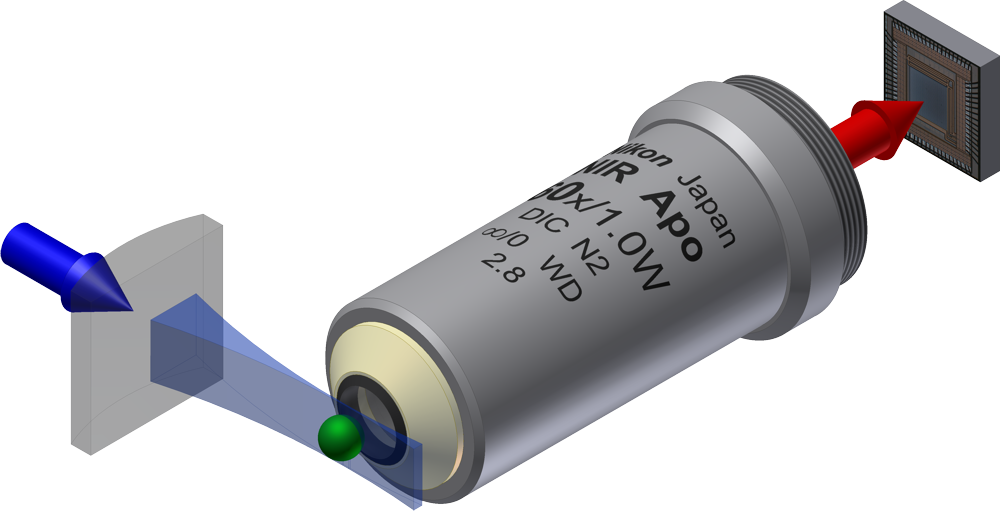} \ \ \ \ 
 \includegraphics[width=0.35\textwidth]{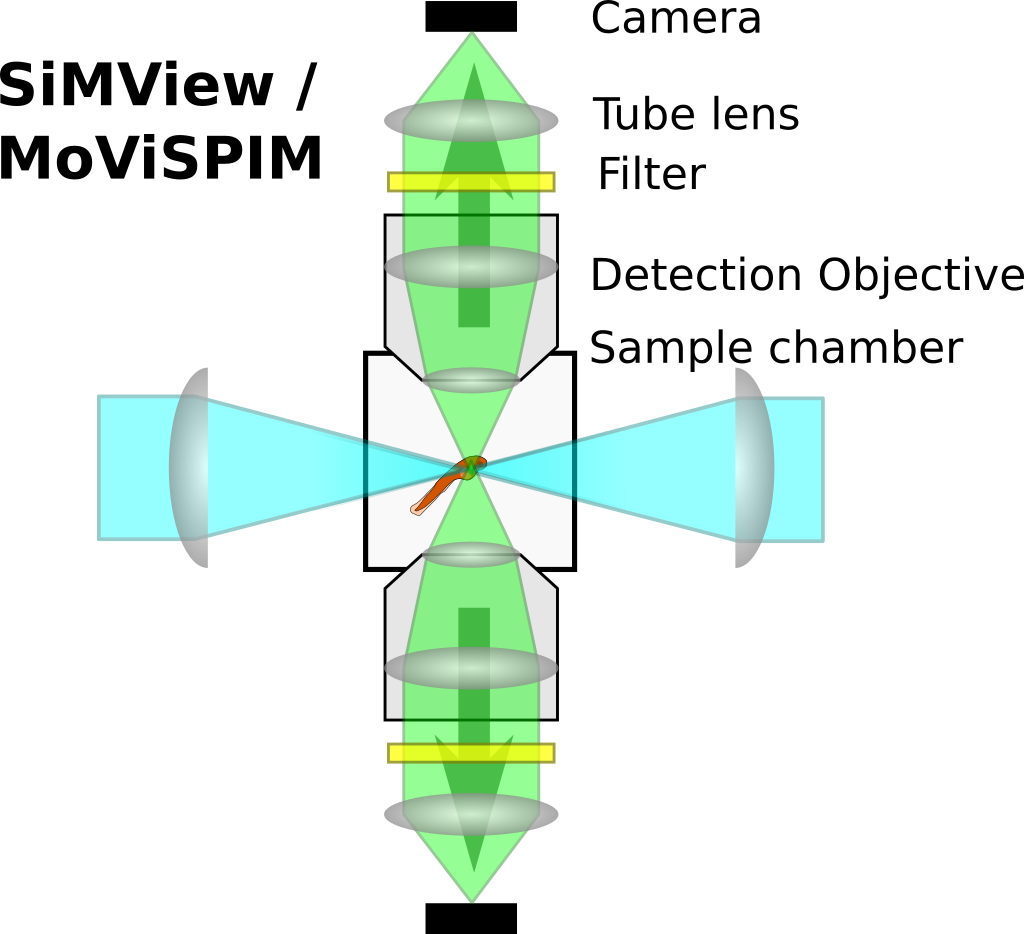}
 \caption{Left: Scheme of a light sheet or SPIM microscope. Right: Scheme of a multiview SPIM. The sample is illuminated from two opposite sides and images are formed in two opposite planes parallel to the light sheet.}\label{fig:SPIM}
\end{figure}

 \begin{figure}
	    \centering
			\begin{tikzpicture}
			\draw[thin]  (0.2,0.2) rectangle (3.8,3.8);

			\draw[thick,red] (0,2.1) -- (1.9,2.1) ;
			\draw[thick,green] (1.9,2.1) -- (1.9,4);
			\draw[->,thick,red] (0,2.1) -- (1,2.1) ;
			\draw[->,thick,green] (1.9,2.1) -- (1.9,3) ;
			\draw (1.5,4) node[above]{$u_{1,+}$};
			
			\draw[thick,red] (0,1.9) -- (1.9,1.9);
			\draw[thick,green] (1.9,1.9) -- (1.9,0);
			\draw[->,thick,red] (0,1.9) -- (1,1.9) ;
			\draw[->,thick,green] (1.9,1.9) -- (1.9,1) ;
			\draw (1.5,0) node[below]{$u_{2,+}$};

			\draw[thick,red] (4,2.1) -- (2.1,2.1);
			\draw[thick,green] (2.1,2.1) -- (2.1,4);
			\draw[->,thick,red] (4,2.1) -- (3,2.1) ;
			\draw[->,thick,green] (2.1,2.1) -- (2.1,3) ;
			\draw (2.5,4) node[above]{$u_{1,-}$};

			\draw[thick,red] (4,1.9) -- (2.1,1.9) ;
			\draw[thick,green] (2.1,1.9) -- (2.1,0) ;
			\draw[->,thick,red] (4,1.9) -- (3,1.9) ;
			\draw[->,thick,green] (2.1,1.9) -- (2.1,1) ;
			\draw (2.5,0) node[below]{$u_{2,-}$};
			\end{tikzpicture}
			\caption{Light propagation for the 4 multiview SPIM images. In red, the excitation path (laser). In green, the emission path (fluorophores).}\label{fig:schememspim}
 \end{figure}
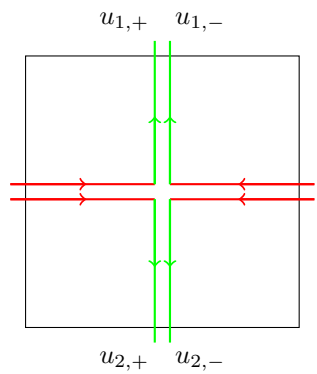

\subsection{Summary}
 
We showed three different applications where measuring attenuation could be attacked with a similar methodology. 
In all mentioned applications, we have access to a set of $m$ signals $(u_i)_{1\leq i \leq m}$ with the following expression:
\begin{equation}\label{eq:mainmodelequation}
 u_i = \mathcal{P}\left( \beta \exp(-\boldsymbol{A}_i \boldsymbol{\alpha})\right),
\end{equation}
where $\boldsymbol{\alpha}$ is a concatenation of attenuation coefficients at different wavelengths and $\boldsymbol{A}$ is a concatenation of integral operators. 
In this paper, we will focus on the simplest setting where only two views are available. 
We developed the algorithms in such a way that their extension to an arbitrary number of views be rather straightforward. 
In addition, we will assume that the two views are opposite in the numerical experiments. This is a requirement in 1D, but not when dealing with 2D or 3D images.

\section{MAP estimator and numerical evaluation}\label{sec:bayes}

\subsection{The discretized model}

The discrete model considered in this paper reads:
\begin{equation}\label{eq:u1u2}
\left\{
\begin{array}{ll}
 u_1 &= \mathcal{P}\left( \beta \exp(-A_1 \alpha)\right)\\
 u_2 &= \mathcal{P}\left( \beta \exp(-A_2 \alpha)\right).
\end{array}\right.
\end{equation}
The signals $u_1$, $u_2$, $\beta$ and $\alpha$ are assumed to be nonnegative and belong $\R^n$, where $n=n_1\hdots n_d $ denotes the number of observations and $d$ is the space dimension. 
The value of a vector $u_1$ at location $i=(i_1,\hdots, i_d)$ will be denoted either $u_1[i]$ or $u_1[i_1,\hdots,i_d]$.
The matrices $A_1$ and $A_2$ in $\R^{n\times n}$ are discretization of linear integral operators. 
In our numerical experiments, the product $A_1 u_1$ represents the cumulative sum of $u_1$ along one direction and the product $A_2u_2$ represents the cumulative sum of $u_2$ in the opposite direction. 
For instance, for a 1D signal, we set:
\begin{equation}
(A_1 u) [i] = \sum_{j=1}^i u_1[j].
\end{equation}
Therefore, matrix $A_1$ has the following lower triangular shape:
\begin{equation}\label{eq:matrixA1}
 A_1 = \begin{pmatrix}
        1 & 0 & 0 & 0 & \hdots & 0\\
        1 & 1 & 0 & 0 & \hdots & 0 \\
        1 & 1 & 1 & 0 & \hdots & 0\\
	   \hdots & \hdots & \hdots & \ddots & \hdots & 0 \\
	   1 & 1 & 1 & 1 & \hdots  & 1 
       \end{pmatrix}
\end{equation}

We are now ready to design a Bayesian estimator of $\alpha$ and $\beta$ from model \eqref{eq:u1u2}.

\subsection{A Bayesian estimator}

The Maximum A Posteriori (MAP) estimators $\hat \alpha$ and  $ \hat \beta$ of $\alpha$ and $\beta$ are defined as the maximizers of the conditional probability density:
\begin{equation}
 \max_{\alpha\in \R^n,\beta \in \R^n}  \P(\alpha,\beta | u_1,u_2).
\end{equation}
By using the Bayes rule and a negative log-likelihood, this is equivalent to finding the minimizers of:
\begin{equation}
 \min_{\alpha\in \R^n,\beta\in \R^n} - \log(\P(u_1,u_2| \alpha,\beta)) - \log(\P(\alpha,\beta)).
\end{equation}
Let us evaluate $\P(u_1,u_2| \alpha,\beta)$. To this end, set 
\begin{equation}
 \lambda_1 = \beta\exp(- (A_1\alpha)) \mbox{ and } \lambda_2 = \beta\exp(- (A_2\alpha)).
\end{equation}
Since the distribution of a Poisson distributed random variable with parameter $\lambda$ has the following probability mass function:
\begin{equation}
 \mathbb{P}(X=k)=\frac{\lambda^k e^{-\lambda}}{k!},
\end{equation}
we get:
\begin{equation}
 \P(u_1,u_2| \alpha,\beta) = \sum_{i=1}^n \lambda_1[i] + \lambda_2[i] - u_1[i]\log(\lambda_1[i])-u_2[i]\log(\lambda_2[i]) + C,
\end{equation}
where $C$ is a value that does not depend on $\alpha$ and $\beta$.
Next, we assume that $\alpha$ and $\beta$ are \emph{independent random vectors} with probability distribution functions of type:
\begin{equation}
 \P(\alpha) \propto \exp(-R_\alpha(\alpha)) \mbox{ and } \P(\beta) \propto \exp(-R_\beta(\beta)),
\end{equation}
where $R_\alpha:\R^n\to \R\cup \{+\infty\}$ and $R_\beta:\R^n\to \R\cup \{+\infty\}$  are regularizers describing properties of the density and attenuation maps. In particular, we choose them so as to impose nonnegativity of the estimators:
\begin{equation}
 R_\alpha(\alpha)=+\infty  \textrm{ if } \exists i\in \{1,\hdots,n\}, \alpha[i]<0
\end{equation}
and 
\begin{equation}
 R_\beta(\beta)=+\infty \textrm{ if } \exists i\in \{1,\hdots,n\}, \beta[i]<0.
\end{equation}

Overall, the optimization problem characterizing the MAP estimates reads:
\begin{equation}\label{eq:mainproblem}
 \min_{\alpha \in \R^n, \beta\in \R^n} F(\alpha,\beta) 
\end{equation}
where 
\begin{equation}\label{eq:defF}
 F(\alpha,\beta)=\sum_{i=1}^n \sum_{j=1}^2[\exp(-(A_j\alpha)[i])\beta[i] +u_j[i]((A_j\alpha)[i]-\log(\beta[i]))] + R_\alpha(\alpha) + R_\beta(\beta).
\end{equation}
\begin{remark}
 With an arbitrary number $m$ of views, it suffices to replace $\sum_{j=1}^2$ by $\sum_{j=1}^m$ in the above expression.
 For $m=1$ view, the problem $\min_{\alpha} F(\alpha,\beta)$ allows recovering the attenuation knowing the density: this is an inverse problem met in lidar. To the best of our knowledge, the proposed formulation is novel for this problem.
 The problem $\min_{\beta} F(\alpha,\beta)$ corresponds to correcting the attenuation on the density map. 
 This is also a frequently met problem \cite{rigaut1991high,roerdink1993fft,can2003attenuation,kervrann2004robust} and the proposed approach also seems novel. 
\end{remark}


\begin{remark}
The hypothesis of independence between $\alpha$ and $\beta$ can probably be improved in some applications. 
For instance, in lidar, it is well known that extinction and backscatter coefficients are strongly related. 
Similarly, strongly absorbing parts of biological specimens are likely to have a specific density of fluorophores.
We prefer stating this independence property, since our aim is to derive a generic algorithm.
\end{remark}

\subsection{The overall optimization algorithm}

If $R_\alpha$ and $R_\beta$ are convex functions, then $F$ is convex in each variable separately. Unfortunately, $F$ is usually nonconvex on the product space $\R_+^n\times \R_+^n$ since the 2D function $f(x,y)=\exp(-x)y + x -\log(y)$ is nonconvex.
Finding global minimizers in reasonable times therefore seems complicated, except for specific regularizers $R_\alpha$ and $R_\beta$ that could compensate for the nonconvexity of the data term. The main observation in this paragraph is that it is possible to find a global minimizer of $F$ when $R_\alpha$ is a standard convex regularizer and $R_\beta$ is just the indicator function of the positive orthant.
Set
\begin{equation}
 R_\beta(\beta) = \iota_{\R_+^n}(\beta)=
 \left\{\begin{array}{ll}
 0 & \textrm{ if } \beta[i]\geq 0, \forall i \in \{1,\hdots,n\}, \\
 +\infty & \textrm{otherwise}.
  \end{array}
 \right.
\end{equation}
With this specific choice, the optimality conditions of problem \eqref{eq:mainproblem} with respect to variable $\beta$ read:
\begin{equation}
    \label{eq:betaMP}
 \beta = \frac{u_1+u_2}{ \exp\left( - A_1 \alpha \right) + \exp\left( - A_2 \alpha \right)}.
\end{equation}
This expression can be seen as a simple estimator of $\beta$ knowing $u_1$, $u_2$ and $\alpha$. 
By replacing this expression in \eqref{eq:defF}, we obtain an optimization problem in variable $\alpha$ only:
\begin{equation}\label{eq:initialguess}
 \min_{\alpha\in\R^n}  \sum_{i=1}^n \sum_{j=1}^2 u_j[i] \left[(A_j \alpha)[i] + \log\left( \sum_{j=1}^2 \exp(-(A_j\alpha)[i])\right)\right] + R_\alpha(\alpha).
\end{equation}
\begin{proposition}
 Problem \eqref{eq:initialguess} is convex for a convex regularizer $R_\alpha$.
\end{proposition}
\begin{proof}
The term $u_j[i](A_j \alpha)[i]$ is linear, hence convex. The term $\log\left( \sum_{j=1}^2 \exp(-(A_j\alpha)[i])\right)$ is the composition of the convex logsumexp function with a linear operator, hence it is convex.
\end{proof}

The above observation motivates using the minimizer of \eqref{eq:initialguess} as an initial guess. Then, it is possible to use an arbitrary convex regularizer $R_\beta$ and to minimize \eqref{eq:mainproblem} using an alternate minimization between $\alpha$ and $\beta$. This idea is captured in Algorithm \ref{alg:main}, it ensures a monotonic decay of the cost function: $F(\alpha_{k+1},\beta_{k+1})\leq F(\alpha_k,\beta_k)$.
However the algorithm does not necessarily converge to a stationary point. This is not a critical issue since we will see later that only 1 iteration is usually preferable.
\begin{algorithm}
\caption{An algorithm to solve problem \eqref{eq:mainproblem}}
\label{alg:main}
\begin{algorithmic}[1]
\State \textbf{Input:} Initial guess $u_1$, $u_2$, $A_1$, $A_2$ and $Nit$.
\State \textbf{Output:} $\hat \alpha$ and $\hat \beta$, estimates of the attenuation and density.
\State Find $\alpha_0$, the minimizer of \eqref{eq:initialguess}.
\State Find $\displaystyle \beta_0=\argmin_{\beta\in \R^n_+} F(\alpha_0,\beta)$.
\For{$k=1$ to $Nit$}
\State Find $\displaystyle \alpha_k=\argmin_{\alpha\in \R^n_+} F(\alpha,\beta_{k-1})$.
\State Find $\displaystyle \beta_k=\argmin_{\beta\in \R^n_+} F(\alpha_{k},\beta)$.
\EndFor
\end{algorithmic}
\end{algorithm}

\subsection{Minimizing the initial convex program: warm start}

We now delve into the numerical resolution of the warm start initialization \eqref{eq:initialguess}. 
First, we need to choose a convex regularizer $R_\alpha$. 
In this paper, we propose to simply use the total variation \cite{rudin1992nonlinear}, which is well known to preserve sharp edges.
Its expression is given by:
\begin{equation}
 R_\alpha(\alpha) = \lambda_\alpha\sum_{i=1}^n \|(\nabla \alpha)[i]\|_2,
\end{equation}
where $\nabla :\R^n \to \R^{dn}$ is a discretization of the gradient and $\lambda_\alpha\geq 0$ is a regularization parameter. We will use the standard discretization proposed in \cite{chambolle2004algorithm} in our numerical experiments.

Problem \eqref{eq:initialguess} is convex, but rather hard to minimize for various reasons listed below. 
First, the vectors $\alpha$ and $\beta$ may be very high dimensional, preventing the use of an arbitrary black-box method. 
Second, the regularizer $R_\alpha$ is non differentiable.
Third, the operators $A_i$ have a spectral norm depending on the dimension $n$, preventing the use of gradient based methods since the Lipschitz constant of the gradient would be too high, see Proposition \ref{prop:spectral}. 
Last, the proximal operator associated to the logsumexp function has no simple analytical formula.

\begin{proposition}\label{prop:spectral}
 Matrix $A_1$ in \eqref{eq:matrixA1} satisfies $\|A_1\|_{2\to 2}\gtrsim n$, where $\|\cdot\|_{2\to 2}$ stands for the spectral norm.
\end{proposition}
\begin{proof}
 \begin{align*}
\|A_1\|_{2\to 2}^2  \geq \left\|A_1 \begin{pmatrix}
                                   1/\sqrt{n} \\ \vdots \\ 1/\sqrt{n}
                                  \end{pmatrix} \right\|_2^2
                    \geq \frac{1}{n} \left\|\begin{pmatrix}
                                                            1 \\ 2 \\ \vdots \\ n
                                                           \end{pmatrix}\right\|_2^2
                    \gtrsim \frac{n^3}{n} =n^2.
 \end{align*}
\end{proof}

A large number of splitting methods have been developed to solve problems of type \eqref{eq:initialguess}, and we refer to the excellent review paper \cite{combettes2011proximal} for an overview. Among them, the Simultaneous Direction Method of Multipliers (SDMM), a variant of the ADMM \cite{fortin2000augmented,ng2010solving} is particularily adapted to the structure of our problem. This algorithm allows solving problems of type:
\begin{equation}\label{eq:SDMMprob}
 \min_{\alpha \in \R^n} g_1(L_1\alpha) + \hdots + g_m(L_m\alpha), 
\end{equation}
where functions $g_i:\R^n\to \R\cup\{+\infty\}$ are convex closed and the operators $L_i:\R^n\to \R^{m_i}$ are linear and such that $Q=\sum_{i=1}^n L_i^TL_i$ is an invertible matrix. The SDMM then takes the algorithmic form described in Algorithm \ref{alg:SDMM}.

\begin{algorithm}
\caption{The SDMM algorithm to solve \eqref{eq:SDMMprob}}
\label{alg:SDMM}
\begin{algorithmic}[1]
 \State \textbf{input:} $Nit$, $\gamma>0$, $(y_{i,0})_{1\leq i \leq m}$, $(z_{i,0})_{1\leq i \leq m}$
 \For{$k=1$ to $Nit$}
 \State $x_k=Q^{-1}\sum_{i=1}^m L_i^T(y_{i,k}-z_{i,k})$
 \For{$i=1$ to $m$}
 \State $s_{i,k}=L_ix_k$.
 \State $y_{i,k+1}=\prox_{\gamma g_i}(s_{i,k}+z_{i,k})$
 \State $z_{i,k+1}=z_{i,k}+s_{i,k}-y_{i,k+1}$
 \EndFor
 \EndFor
\end{algorithmic}
\end{algorithm}

To cast problem \eqref{eq:initialguess} into form \eqref{eq:SDMMprob}, we use the following choices.
We set 
 \begin{equation}\label{eq:defL1}
  \begin{array}{lll}
      L_1:&\R^n&\to \R^{2n}     \\
      &\alpha &\mapsto c_1\begin{pmatrix}
                       A_1 \alpha \\
                       A_2 \alpha
                      \end{pmatrix}
    \end{array},
 \end{equation}
 \begin{equation}
   \begin{array}{lll}
    g_1 & : \R^{2n} &\to \R \cup \{+\infty\} \\
    & \begin{pmatrix}
      z_1 \\
      z_2
    \end{pmatrix} 
    &\mapsto  \sum_{i=1}^n \sum_{j=1}^2 u_j[i] \left[z_j[i]/c_1 + \log\left( \sum_{j=1}^2 \exp \left(-z_j[i]/c_1\right)\right)\right],
   \end{array}
 \end{equation}
 \begin{equation}
  L_2 = c_2\nabla  \mbox{ and } 
  g_2 
  \begin{pmatrix}
  z_1 \\ \vdots  \\ z_d
  \end{pmatrix}
  = \frac{\lambda}{c_2} \sum_{i=1}^n \sqrt{z_1^2[i] + \hdots + z_d[i]^2},
 \end{equation}
 \begin{equation}
  L_3 = c_3I_n \mbox{ and } g_3(z) = \iota_{\R_+^n}(z).
 \end{equation}
 The numbers $c_1,c_2,c_3$ are positive constants allowing to accelerate the algorithm's convergence by balancing the relative importance of each term. This can also be seen as a simple diagonal preconditioner.
 In our numerical experiments, we set $c_1=1$ and tune $c_2$ and $c_3$ manually to accelerate convergence.
 
 In order to apply Algorithm \ref{alg:SDMM}, we need to compute the proximal operators of each function $g_i$, defined by:
 \begin{equation}
  \prox_{\gamma g_i} (z_0) = \argmin_{z\in \R^{m_i}} \gamma g_i(z) + \frac{1}{2}\|z-z_0\|_2^2.
 \end{equation}
 The proximal operators of $g_2$ and $g_3$ have closed form solutions found in nearly all recent total variation minimization solvers. We refer to \cite{ng2010solving} for instance.
 Unfortunately, the proximal operator of $g_1$ has no closed-form expression. 
 In order to compute it, we propose using a non trivial Newton based algorithm described in section \ref{appendix:newton1}. 
 Finally, we need to evaluate matrix-vector products with $Q^{-1}$. 
 This can be achieved using either a LU factorization or a conjugate gradient.
 In our codes, we simply use a conjugate gradient algorithm.

 To conclude this paragraph, we illustrate the results obtained by the described procedure in Figure \ref{fig:XP2}.
 
 \begin{figure}
 \centering
  \begin{subfigure}[b]{0.35\textwidth}
    \includegraphics[width=\textwidth]{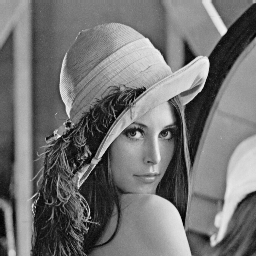}
    \caption{Density $\beta\in[0,100]^n$ }\label{fig:XP2beta}
  \end{subfigure}
  \begin{subfigure}[b]{0.35\textwidth}
    \includegraphics[width=\textwidth]{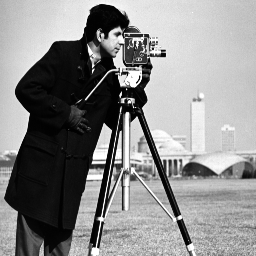}
    \caption{Attenuation $\alpha\in[0,0.03]^n$}\label{fig:XP2alpha}
  \end{subfigure}
  \begin{subfigure}[b]{0.35\textwidth}
    \includegraphics[width=\textwidth]{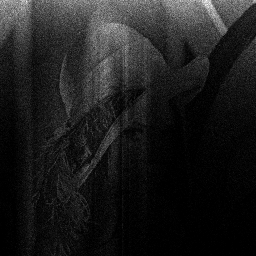}
    \caption{Image $u_1$}\label{fig:XP2u1}
  \end{subfigure}
  \begin{subfigure}[b]{0.35\textwidth}
    \includegraphics[width=\textwidth]{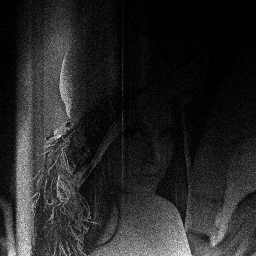}
    \caption{Image $u_2$}\label{fig:XP2u2}
  \end{subfigure}
  \begin{subfigure}[b]{0.35\textwidth}
    \includegraphics[width=\textwidth]{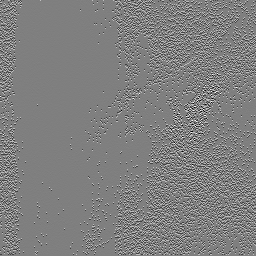}
    \caption{\eqref{eq:betaMP}: SNR=-53.1dB}\label{fig:XP2direct}
  \end{subfigure}
  \begin{subfigure}[b]{0.35\textwidth}
    \includegraphics[width=\textwidth]{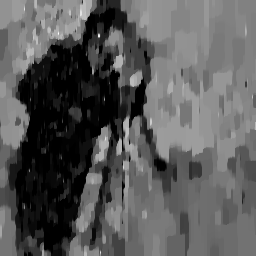}
    \caption{TV: SNR=11.9dB}\label{fig:XP2TV}
  \end{subfigure}  
  \caption{Warm start initialization. \eqref{fig:XP2beta} and \eqref{fig:XP2alpha} are the original density and attenuation. \eqref{fig:XP2u1} and \eqref{fig:XP2u2} are the observed signals. \eqref{fig:XP2direct} is the direct density estimate \eqref{eq:directinversionformula}. As can be seen, the formula yields useless results since it is completely unstable to noise. \eqref{fig:XP2TV} is the density estimate using the total variation solver. It allows recovering the main details of the cameraman, despite a significant amount of noise.}\label{fig:XP2}
\end{figure}

\subsection{Recovering the attenuation knowing the density}\label{sec:attenuationfromdensity}

In this paragraph, we focus on the resolution of:
\begin{align*}
\alpha_k& =\argmin_{\alpha\in \R^n} F(\alpha,\beta_{k-1}) \\
&= \argmin_{\alpha\in \R^n} h_1(L_1\alpha)+h_2(L_2\alpha) + h_3(L_3 \alpha),
\end{align*}
where $L_1$, $L_2$ and $L_3$ are defined as in the previous section, $h_2=g_2$, $h_3=g_3$ and
\begin{equation}
   \begin{array}{lll}
    h_1 & : \R^{2n} &\to \R \\
    & \begin{pmatrix}
      z_1 \\
      z_2
    \end{pmatrix} 
    &\mapsto  \sum_{j=1}^2\sum_{i=1}^n  u_j[i] z_j[i]/c_1 + \exp(-z_j[i]/c_1)\beta_{k-1}[i].
   \end{array}
 \end{equation}
 As in the previous section, the SDMM is a good candidate to solve this problem. 
 The only difficulty is to compute the proximal operator $\prox_{\gamma h_1}$. This amounts to solving $2n$ 1D optimization problems of the following type:
 \begin{equation}\label{eq:proxh1}
  \argmin_{z\in \R} \gamma u z + \gamma \exp(-z) \beta + \frac{1}{2} (z-z_0)^2,
 \end{equation}
 where $u$, $z_0$ and $\beta$ are real numbers.
 \begin{proposition}
  The solution $z^*$ of problem \eqref{eq:proxh1} is given by:
  \begin{equation}
    z^* = a +W(\beta\gamma \exp(-a)),
  \end{equation}
  where $a=z_0 - \gamma u$ and $W$ is the so-called Lambert W function, i.e. the reciprocal function of $x\mapsto x\exp(x)$. 
 \end{proposition}
\begin{proof}
 The optimality conditions for this problem read:
 \begin{equation}
  \gamma u- \beta\gamma \exp(-z^*) + z^* -z_0=0,
 \end{equation}
 so that
 \begin{equation}
  z^* = a + \beta\gamma \exp(-z^*),
 \end{equation}
 with $a=z_0-\gamma u$. Now, we can write $z^* = a + w$, hence $w =  \beta\gamma  \exp(-a - w)$, which is still equivalent to:
 \begin{equation}
  w\exp(w) = \beta\gamma \exp(-a),
 \end{equation}
 or $w = W(\beta\gamma \exp(-a))$.
\end{proof}
\begin{remark}
  The Lambert W function has been studied thoroughly. The standard approach to compute it consists of using Halley's method (a Househ\"older method with cubic convergence rate) initialized with the first terms of an asymptotic expansion. This method was proposed in the excellent review paper \cite{corless1996lambertw} and seems to be the default solver in MAPLE and MATLAB. Unfortunately, this method fails for our problem since the number $\beta\gamma \exp(-a)$ can exceed the maximum number available in double precision. We therefore develop a specific method in Appendix \ref{appendix:lambertw}.
\end{remark}

\subsection{Recovering the density knowing the attenuation}\label{sec:densityfromattenuation}
In this paragraph, we focus on the resolution of $\min_{\beta \in \R^n} F(\alpha_{k},\beta)$. 
This amounts to simultaneously correcting the attenuation and denoising the resulting image. 
This is a rather simple inverse problem, but it seems original due to the noise statistics. 
A Poisson distributed variable multiplied by a positive constant different from 1 is not Poisson anymore.
This makes the proposed algorithm similar, but different from existing approaches developed for Poisson noise in \cite{dupe2009proximal,steidl2010removing} for instance.

A simple idea to regularize the problem is to use the total variation again, i.e. to set $R_\beta(\beta) = \lambda_\beta\sum_{i=1}^n \|(\nabla \beta)[i]\|_2$. Once again the resulting problem can be solved with the SDMM. Let us detail this procedure. Define 
\begin{equation}
 a[i]=\sum_{j=1}^2 \exp(-(A_j\alpha)[i]) \mbox{ and } u[i]=\sum_{j=1}^2 u_j[i].
\end{equation}
The problem then reads:
\begin{align*}
 & \min_{\beta\in \R^n_+} \sum_{i=1}^n a[i]\beta[i]-u[i]\log(\beta[i]) + \lambda_\beta\sum_{i=1}^n \|(\nabla \beta)[i]\|_2 \\
 &=\min_{\beta\in \R^n} j_1(L_1\beta) + j_2(L_2\beta),
\end{align*}
with $L_1=c_1 I_n$, 
 \begin{equation}
   \begin{array}{lll}
    j_1 & : \R^{n} &\to \R \cup \{+\infty\} \\
    & z &\mapsto \iota_{\R_+^n}(z) + \frac{1}{c_1}\sum_{i=1}^n  a[i]z[i]-u[i]\log(z[i]),
   \end{array}
 \end{equation}
 $L_2=c_2 \nabla$ and
 \begin{equation}
    \begin{array}{lll}
    j_2 & : \R^{2n} &\to \R \\
    & \begin{pmatrix}
      z_1 \\
      z_2
    \end{pmatrix} 
    &\mapsto  \frac{\lambda_\beta}{c_2} \sum_{i=1}^n  \sqrt{z_1[i]^2+z_2[i]^2}.
   \end{array}
 \end{equation}

The proximal operators of $j_2$ is standard and we do not detail it here. The proximal operator of $j_1$ is provided below:
\begin{proposition}
 We have:
 \begin{equation}\label{eq:proxj1}
  \prox_{\gamma j_1} (z_0)= \frac{-(\gamma/c_1 a - z_0) + \sqrt{(\gamma/c_1 a - z_0)^2 + 4 \gamma u}}{2}.
 \end{equation}
\end{proposition}
\begin{proof}
 It suffices to write the first order optimality conditions of $\min_{z\geq 0} 1/2\|z-z_0\|_2^2 + a/c_1 z - u\log(z/c_1)$. This shows that $z$ is the root of a second order polynomial. Its only positive root is given in \eqref{eq:proxj1}.
\end{proof}

We show a typical result of total variation minimization in Figure \eqref{fig:XP3}. Pamareter $\lambda_\beta$ was chosen manually so as to maximize the SNR of the result. 

\begin{figure}
 \centering
  \begin{subfigure}[b]{0.35\textwidth}
    \includegraphics[width=\textwidth]{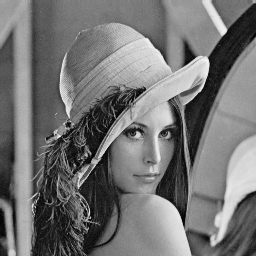}
    \caption{Density $\beta\in[0,100]^n$ }\label{fig:XP3beta}
  \end{subfigure}
  \begin{subfigure}[b]{0.35\textwidth}
    \includegraphics[width=\textwidth]{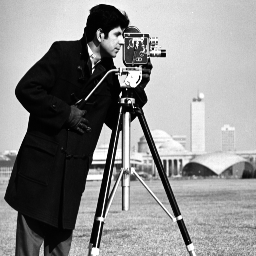}
    \caption{Attenuation $\alpha\in[0,0.03]^n$}\label{fig:XP3alpha}
  \end{subfigure}
  \begin{subfigure}[b]{0.35\textwidth}
    \includegraphics[width=\textwidth]{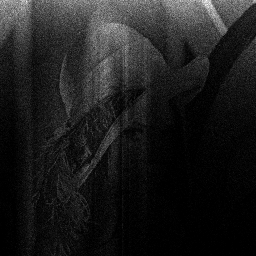}
    \caption{Image $u_1$}\label{fig:XP3u1}
  \end{subfigure}
  \begin{subfigure}[b]{0.35\textwidth}
    \includegraphics[width=\textwidth]{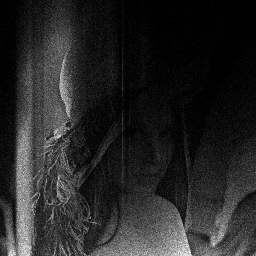}
    \caption{Image $u_2$}\label{fig:XP3u2}
  \end{subfigure}
  \begin{subfigure}[b]{0.35\textwidth}
    \includegraphics[width=\textwidth]{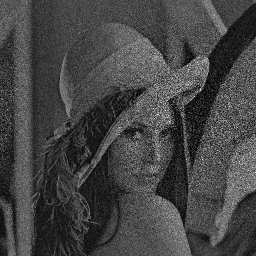}
    \caption{\eqref{eq:betaMP}: SNR=12.9dB}\label{fig:XP3MLE}
  \end{subfigure}
  \begin{subfigure}[b]{0.35\textwidth}
    \includegraphics[width=\textwidth]{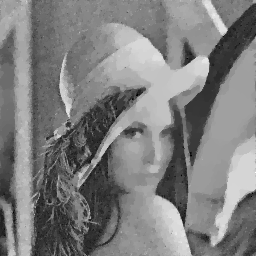}
    \caption{TV: SNR=21.9dB}\label{fig:XP3TV}
  \end{subfigure}  
  \caption{Recovering the density knowing the exact attenuation, with a non regularized estimator or a total variation solver. \eqref{fig:XP3beta} and \eqref{fig:XP3alpha} are the original density and attenuation. \eqref{fig:XP3u1} and \eqref{fig:XP3u2} are the observed signals. \eqref{fig:XP3MLE} is the direct density estimate \eqref{eq:betaMP}. \eqref{fig:XP3TV} is the density estimate using a total variation solver.}\label{fig:XP3}
\end{figure}

\section{Additional comments}\label{sec:numerical}

\subsection{Parameter selection}

\paragraph{Data terms}

The two data term parameters are $\lambda_\alpha$ and $\lambda_\beta$. They specify the regularity of the attenuation and the density respectively.
In all our experiments, we optimized them by trial and error. We observed experimentally, that similar results are obtained within a relatively large range, making a manual optimization quite easy. In addition, for a given measurement device, the same parameter is likely to be always the same, decreasing the interest of an automatized procedure such as SURE. 

\paragraph{Algorithms parameters}

The optimization algorithms are based on the SDMM and their convergence rates depend a lot on the paramaters $\gamma$, $c_1$, $c_2$, $c_3$ and $c_4$. 
They may converge to a satisfactory solution rapidly (about 50 iterations) or slowly (more than 10000 iterations) depending on these choices.
Unfortunately, we found no systematic method to choose them and also used a trial and error strategy in our numerical experiments. 
Our numerical experiments suggest that these parameters are suitable for a wide range of data (image size, maximum attenuation, image dynamics), so that the tedious tuning can be done once for all for a given application.

\subsection{Number of outer iterations}

Algorithm \ref{alg:main} depends on a number of outer iterations denoted $Nit$. One may wonder how many iterations are needed to obtain a satisfactory solution. It turns out that 1 iteration leads to the best results in terms of SNR: iterating more tends to degrade the solution, even though the cost function continues to decrease. The precise reason behind this is still unknown, but it is likely that this is total variation bias. Low contrasted solutions are favored at the expense of SNR. This seems to be one more example, where the Maximum A Posteriori principle should be taken with caution \cite{nikolova2007model,gribonval2011should,starck2013sparsity}. 
It is useful to get a rough idea of a functional to minimize (here it allowed designing the warm start estimate), but should not be considered as the best possible estimator. 

This remark being stated, the proposed algorithm can be seen as a simple two step procedure:
\begin{itemize}
 \item Find the warm start estimate \eqref{eq:initialguess}.
 \item Correct the density (see section \ref{sec:densityfromattenuation}).
\end{itemize}
It is possible to evaluate the attenuation once again (see section \ref{sec:attenuationfromdensity}), but we observed that this influenced the result very little. 


\subsection{Computing times}

All the experiments of the paper were performed on a laptop with an Intel i7 processor with 4-cores. 
The codes were written mostly in Matlab (natively parallel), with some parts written in C with OpenMP support.

The complexity of the proposed algorithms scale roughly linearly with the number of pixels $n$, as shown in Fig. \ref{fig:dependenceonn}.
We observed that the number of iterations of the SDMM to reach a given relative accuracy remains the same whatever the size $n$, while the cost per 
iteration scales linearly with it (at least for the cumulative sum integral operators considered herein). 

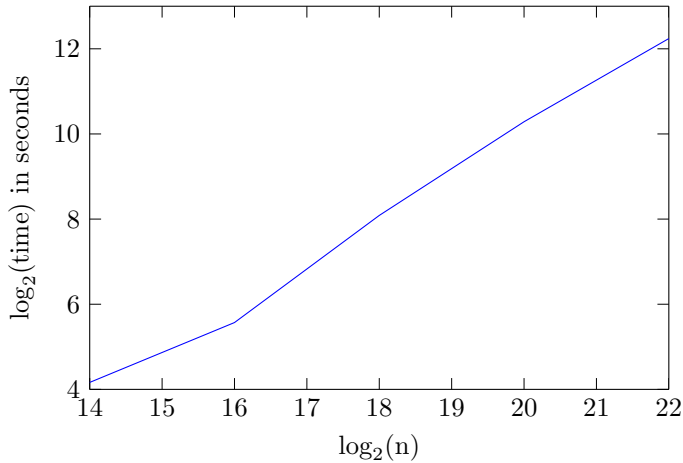
\begin{figure}
%
%
\begin{tikzpicture}

\begin{axis}[%
width=3in,
height=2in,
at={(0.809in,0.513in)},
scale only axis,
separate axis lines,
every outer x axis line/.append style={black},
every x tick label/.append style={font=\color{black}},
every x tick/.append style={black},
xmin=14,
xmax=22,
xlabel={$\text{log}_\text{2}\text{(n)}$},
every outer y axis line/.append style={black},
every y tick label/.append style={font=\color{black}},
every y tick/.append style={black},
ymin=4,
ymax=13,
ylabel={$\text{log}_\text{2}\text{(time)}$ in seconds},
axis background/.style={fill=white}
]
\addplot [color=blue, forget plot]
  table[row sep=crcr]{%
14	4.16349076201482\\
16	5.56957241746771\\
18	8.08585114049578\\
20	10.2866840594678\\
22	12.240966908493\\
};
\end{axis}
\end{tikzpicture}%
\caption{Time needed to compute the warm start estimate and correct the attenuation with respect to the number $n$ of pixels (in log-log scale). A linear regression indicates that the slope is roughly equal to 1, showing a linear dependency with respect to the number of pixels.}\label{fig:dependenceonn}
\end{figure}

As can be seen on Fig. \ref{fig:dependenceonn} the algorithm takes around 48 seconds for a $256\times 256$ image. 
Out of these, 45 seconds are spent to recover the attenuation, while the 3 remaining are dedicated to correct the density.

All codes can be easily parallelized on a GPU. A speed-up of 100 can be expected on such an architecture, making the proposed methods suitable for large 2D or 3D images.

\subsection{Influence of attenuation and signal-to-noise ratio}

Two parameters strongly influence the ability to recover the attenuation and density: the signals dynamics (or signal-to-noise ratio) and the attenuation dynamics. 

As the signal-to-noise ratio decreases, it becomes impossible to recover fine details. For instance, the fine stripes are not recovered in Fig. \ref{fig:alpha}, but they are recovered for signals with a much higher amplitude. In Fig. \ref{fig:SNRVSparameters}, it can indeed be verified that a high dynamics of $10^5$ allows recovering most of the stripes. 
This experiment shows that highly sensitive EMCCD cameras should be preferred over more standard devices for this specific application. 

The attenuation amplitude also plays a key role: if it is too low, then no attenuation can be detected. On the contrary, if it is too high, then the signals $u_1$ and $u_2$ will vanish too rapidly, making it impossible to evaluate the attenuation. This is illustrated in Fig. \ref{fig:SNRVSparameters}. It is remarkable that the algorithm manages to recover the attenuation partially for very low signal-to-noise ratio. In Fig. \ref{fig:XP13}, we observe that the attenuation is partially recovered with no more than 30 expected photons per pixel!

\begin{figure}
 \centering
  \begin{subfigure}[b]{0.2\textwidth}
    \includegraphics[width=\textwidth]{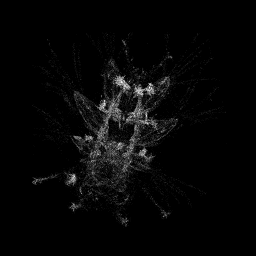}
    \caption{$\|\beta\|_\infty=30$}\label{fig:XP11}
  \end{subfigure}
  \begin{subfigure}[b]{0.2\textwidth}
    \includegraphics[width=\textwidth]{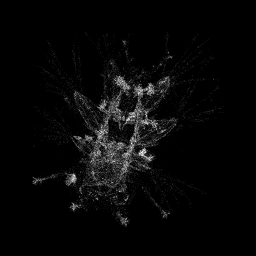}
    \caption{$\|\alpha\|_\infty=0.01$}\label{fig:XP12}
  \end{subfigure}
  \begin{subfigure}[b]{0.2\textwidth}
    \includegraphics[width=\textwidth]{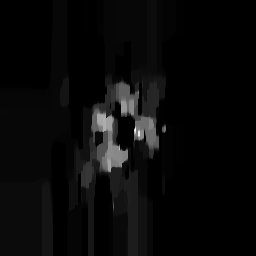}
    \caption{SNR: 3.3dB}\label{fig:XP13}
  \end{subfigure}
  \begin{subfigure}[b]{0.2\textwidth}
    \includegraphics[width=\textwidth]{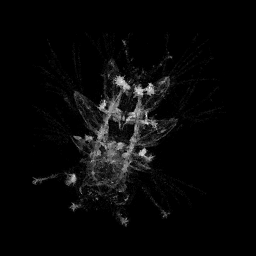}
    \caption{SNR: 15.1dB}\label{fig:XP14}
  \end{subfigure}
  
    \begin{subfigure}[b]{0.2\textwidth}
    \includegraphics[width=\textwidth]{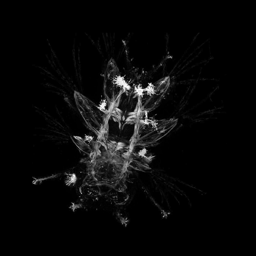}
    \caption{$\|\beta\|_\infty=10^4$}\label{fig:XP21}
  \end{subfigure}
  \begin{subfigure}[b]{0.2\textwidth}
    \includegraphics[width=\textwidth]{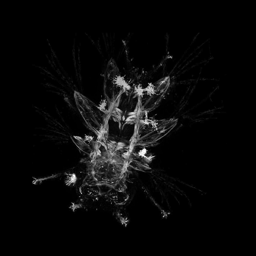}
    \caption{$\|\alpha\|_\infty=0.01$}\label{fig:XP22}
  \end{subfigure}
  \begin{subfigure}[b]{0.2\textwidth}
    \includegraphics[width=\textwidth]{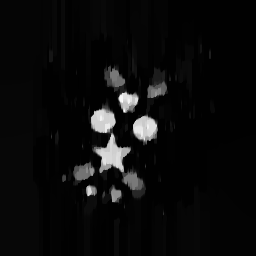}
    \caption{SNR: 9.2dB}\label{fig:XP23}
  \end{subfigure}
  \begin{subfigure}[b]{0.2\textwidth}
    \includegraphics[width=\textwidth]{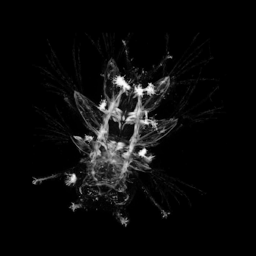}
    \caption{SNR: 21.2dB}\label{fig:XP24}
  \end{subfigure}
  
    \begin{subfigure}[b]{0.2\textwidth}
    \includegraphics[width=\textwidth]{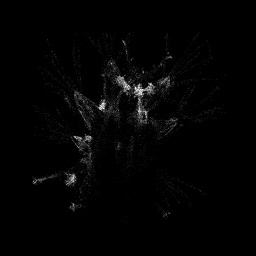}
    \caption{$\|\beta\|_\infty=30$}\label{fig:XP31}
  \end{subfigure}
  \begin{subfigure}[b]{0.2\textwidth}
    \includegraphics[width=\textwidth]{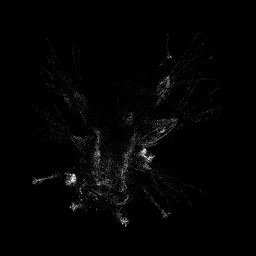}
    \caption{$\|\alpha\|_\infty=0.15$}\label{fig:XP32}
  \end{subfigure}
  \begin{subfigure}[b]{0.2\textwidth}
    \includegraphics[width=\textwidth]{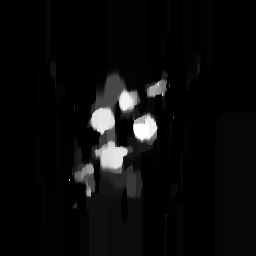}
    \caption{SNR: 6.0dB}\label{fig:XP33}
  \end{subfigure}
  \begin{subfigure}[b]{0.2\textwidth}
    \includegraphics[width=\textwidth]{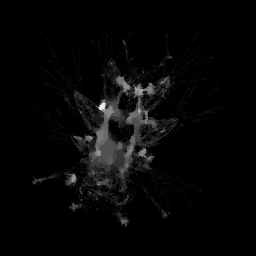}
    \caption{SNR: 9.7dB}\label{fig:XP34}
  \end{subfigure}
  
    \begin{subfigure}[b]{0.2\textwidth}
    \includegraphics[width=\textwidth]{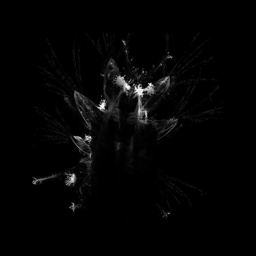}
    \caption{$\|\beta\|_\infty=10^4$}\label{fig:XP41}
  \end{subfigure}
  \begin{subfigure}[b]{0.2\textwidth}
    \includegraphics[width=\textwidth]{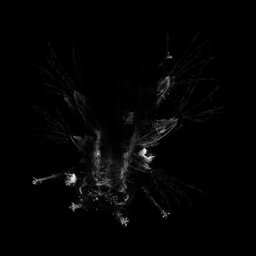}
    \caption{$\|\alpha\|_\infty=0.15$}\label{fig:XP42}
  \end{subfigure}
  \begin{subfigure}[b]{0.2\textwidth}
    \includegraphics[width=\textwidth]{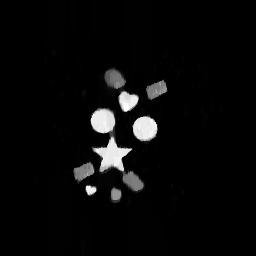}
    \caption{SNR: 11.2dB}\label{fig:XP43}
  \end{subfigure}
  \begin{subfigure}[b]{0.2\textwidth}
    \includegraphics[width=\textwidth]{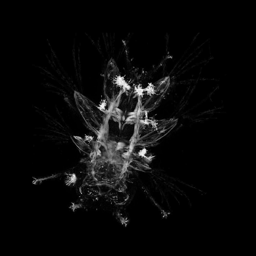}
    \caption{SNR: 21.1dB}\label{fig:XP44}
  \end{subfigure}
  
    \begin{subfigure}[b]{0.2\textwidth}
    \includegraphics[width=\textwidth]{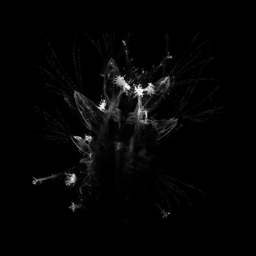}
    \caption{$\|\beta\|_\infty=10^5$}\label{fig:XP51}
  \end{subfigure}
  \begin{subfigure}[b]{0.2\textwidth}
    \includegraphics[width=\textwidth]{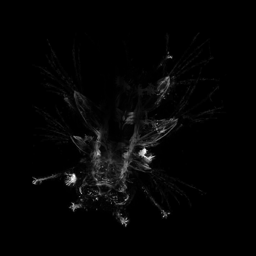}
    \caption{$\|\alpha\|_\infty=0.1$}\label{fig:XP52}
  \end{subfigure}
  \begin{subfigure}[b]{0.2\textwidth}
    \includegraphics[width=\textwidth]{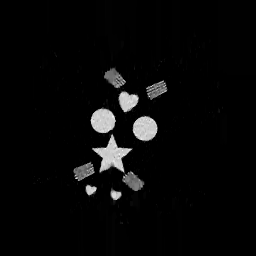}
    \caption{SNR: 13.5}\label{fig:XP53}
  \end{subfigure}
  \begin{subfigure}[b]{0.2\textwidth}
    \includegraphics[width=\textwidth]{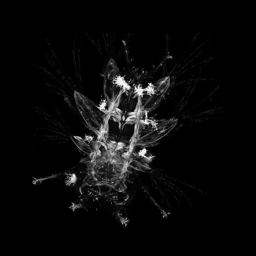}
    \caption{SNR: 44.3}\label{fig:XP54}
  \end{subfigure}  
  
  \caption{Ability to recover the attenuation and density depending on the density and attenuation amplitude. First and second column: attenuated images given as input to the algorithm. Third column: recovered attenuation $\alpha$. Fourth column: recovered denisty $\beta$.}\label{fig:SNRVSparameters}
\end{figure}

\subsection{Toolbox}

A Matlab toolbox containing all the main algorithms described in the paper is provided on the website of the authors \url{https://www.math.univ-toulouse.fr/~weiss/} and on GitHub \url{https://github.com/pierre-weiss/MAEC}. The Lambert W function and the proximal operator of logsumexp have been implemented with C-mex files with OpenMP support for multicore acceleration. Demonstration scripts are available for testing.

\section{Conclusion \& outlook}\label{sec:conclusion}

We proposed a robust and efficient approach to recover attenuation and correct density from multiview measurements. 
This principle was already known in the field of lidar and solved with simple filtering approaches. 
The algorithms proposed herein are based on a clear and versatile statistical framework. 
The approach seems promising in various devices such as lidar or some fluorescence microscopes. 
It is likely that its scope be much wider and we therefore provide a free Matlab toolbox on our website.

As a prospective, we plan to confront our algorithms with real data coming from lidar and microscopy. 
The total variation based algorithm to correct attenuation defects is somewhat disappointing since it is unable to recover fine textures. 
A promising and unexplored issue is to extend the nonlocal means algorithms \cite{nlm} to solve this problem. 
To conclude, let us mention a serious limitation of the proposed approach: 
it is not so common to find a couple optical system-sample, where attenuation dominates scattering. 
We do not know at the present time how many applications can reasonably be modeled by equation \eqref{eq:mainmodelequation}. 
This question is central to precisely understand the strengths and limits of the proposed approach. 

\section{Appendices}

\subsection{Proximal operator of logsumexp in dimension 2}\label{appendix:newton1}

In this section, we propose a fast and accurate numerical algorithm based on Newton's method to solve the following problem:
\begin{align*}
w&=\prox_{\gamma g_1}(z) \\
&=\argmin_{x\in \R^{2n}} \gamma \sum_{i=1}^n \sum_{j=1}^2 u_j[i] \left[x_j[i] + \log\left( \sum_{j=1}^2 \exp(-x_j[i])\right)\right] + \frac{1}{2}\|x-z\|_2^2,
\end{align*}
where 
$z=\begin{pmatrix}
z_1 \\z_2
\end{pmatrix}$ 
and 
$x=\begin{pmatrix}
x_1 \\x_2
\end{pmatrix}$ 
are vectors in $\R^{2n}$. This problem may seem innocuous at first sight, but turns out to be quite a numerical challenge.
The first observation is that it can be decomposed as $n$ \emph{independent} problems of dimension 2 since:
\begin{equation}\label{eq:dirtyproblem}
w[i]=\argmin_{(x_1,x_2)\in \R^{2}} \gamma \sum_{j=1}^2 u_j[i] \left[x_j + \log\left( \sum_{j=1}^2 \exp(-x_j)\right)\right] + \frac{1}{2}(x_j-z_j[i])_2^2.
\end{equation}
To simplify the notation, we will skip the index $i$ in what follows. The following proposition shows that our problem is equivalent to finding the proximal operator associated to the ``logsumexp'' function. 
\begin{proposition}
Define the logsumexp function $\lse(x_1,x_2)=\log\left( \sum_{j=1}^2 \exp(x_j)\right)$.
The solution of problem \eqref{eq:dirtyproblem} coincides with the opposite of the proximal operator of $\lse$:
 \begin{align}
  w[i] &= -\argmin_{(x_1,x_2)\in \R^2} a \lse(x_1,x_2)+ \frac{1}{2}((x_1-y_1)^2+(x_2-y_2)^2) \\
  &= -\prox_{a \lse} (y_1,y_2), \label{eq:proxlse}
 \end{align}
where $a=\gamma(u_1+u_2)$ and $y_j = \gamma u_j - z_j$.
\end{proposition}
\begin{proof}
The first order optimality conditions for problem \eqref{eq:dirtyproblem} read
\begin{equation}
\left\{
 \begin{array}{l}
  \gamma u_1 - \frac{\gamma(u_1+u_2)\exp(-x_1)}{\exp(-x_1)+\exp(-x_2)} + x_1 - z_1 = 0 \\
  \gamma u_2 - \frac{\gamma(u_1+u_2)\exp(-x_2)}{\exp(-x_1)+\exp(-x_2)} + x_2 - z_2 = 0.
 \end{array}\right.
\end{equation}
By letting $a=\gamma(u_1+u_2)$ and $y_j = \gamma u_j - z_j$, this equation becomes
\begin{equation}
\left\{
 \begin{array}{l}
   - \frac{a\exp(-x_1)}{\exp(-x_1)+\exp(-x_2)} + x_1 + y_1 = 0 \\
   - \frac{a\exp(-x_2)}{\exp(-x_1)+\exp(-x_2)} + x_2 + y_2  = 0.
 \end{array}\right.
\end{equation}
It now suffices to make the change of variable $x'_i=-x_i$ to retrieve the optimality conditions of problem \eqref{eq:proxlse}
\begin{equation}\label{eq:optcondlse}
\left\{
 \begin{array}{l}
   \frac{a\exp(x'_1)}{\exp(x'_1)+\exp(x'_2)} + x'_1 - y_1 = 0 \\
   \frac{a\exp(x'_2)}{\exp(x'_1)+\exp(x'_2)} + x'_2 - y_2  = 0.
 \end{array}\right.
\end{equation}
\end{proof}
\begin{remark}
To the best of our knowledge, this is the first attempt to find a fast algorithm to evaluate the prox of logsumexp.
This function is important in many regards. 
In particular, it is a $C^\infty$ approximation of the maximum value of a vector. 
In addition, its Fenchel conjugate coincides with the Shannon entropy restricted to the unit simplex.
We refer to \cite[\S 3.2]{hiriart2006note} for some details.
The algorithm that follows has potential applications outside the scope of this paper.
\end{remark}

We now design a fast and accurate minimization algorithm for problem \eqref{eq:proxlse} or equivalently, a root finding algorithm for problem \eqref{eq:optcondlse}. 
This algorithm differs depending on whether $y_1\geq y_2$ or $y_2\geq y_1$.
We focus on the case $y_1\geq y_2$. The case $y_2\geq y_1$ can be handled by symmetry.

Let $\lambda = \frac{\exp(x'_1)}{\exp(x'_1)+\exp(x'_2)}$ and notice that $\frac{\exp(x'_2)}{\exp(x'_1)+\exp(x'_2)}=1-\lambda$. Therefore \eqref{eq:optcondlse} becomes:
\begin{equation}
\left\{\begin{array}{l}
x'_1 = y_1 -a\lambda \\
x'_2 = y_2 - a(1-\lambda).
 \end{array}\right.
\end{equation}
Hence
\begin{equation}
\frac{1-\lambda}{\lambda} = \exp(x'_2-x'_1) = \exp(y_2-y_1-a)\exp(2a\lambda).
\end{equation}
Taking the logarithm on each side yields \footnote{Applying the logarithm is important for numerical purposes. When $y_2-y_1-a$ is very small, the exponential cannot be computed accurately in double precision.}:
\begin{equation}
\log(1-\lambda)-\log(\lambda) = y_2-y_1-a + 2a\lambda.
\end{equation}
We are now facing the problem of finding the root $\lambda^*$ of the following function:
\begin{equation}
 f(\lambda) = y_2-y_1-a + 2a\lambda - \log(1-\lambda)+\log(\lambda).
\end{equation}
There are two important advantages for this approach compared to the direct resolution of \eqref{eq:optcondlse}.
First, we have to solve a 1D problem instead of a 2D problem. 
More importantly, we directly constrain $x'$ to be of form $x'=y - a\delta$, where $\delta$ lives on the 2D simplex.

Let us collect a few properties of function $f$. First, we have:
\begin{equation}
 f'(\lambda) = 2a + \frac{1}{1-\lambda} + \frac{1}{\lambda} > 0, \forall \lambda\in (0,1).
\end{equation}
Therefore, $f$ is increasing on $(0,1)$. 
To use convergence results of Newton's algorithm, we need to compute $f''$ as well:
\begin{equation}\label{eq:fp}
 f''(\lambda) = -\frac{1}{\lambda^2}+\frac{1}{(1-\lambda)^2}.
\end{equation}

\begin{proposition}\label{prop:encadrement}
If $y_1\geq y_2$, then $x_1'\geq x_2'$ and 
\begin{equation}\label{eq:boundlse}
\max\left( \frac{1}{2}, \frac{1}{1+\exp(y_2-y_1+a)}\right)\leq \lambda^* \leq  \frac{1}{1+\exp(y_2-y_1)}. 
\end{equation}
\end{proposition}
\begin{proof}
 The first statement can be proven by contradiction. 
 Assume that $x_2'> x_1'$, then equation \eqref{eq:optcondlse} indicates that $y_2>y_1$.
 
 For the second statement, it suffices to evaluate $f$ at the extremities of the interval since $f'>0$.
 We get $f(1/2)=y_2-y_1\leq 0$ and $f\left(\frac{1}{1+\exp(y_2-y_1)}\right) = -a + \frac{2a}{1 + \exp(y_2-y_1)} \geq 0$.
\end{proof}
\begin{proposition}
Set $\lambda_0 = \frac{1}{1+\exp(y_2-y_1)}$. 
Then, the following Newton's method 
\begin{equation}
 \lambda_{k+1} = \lambda_k - \frac{f(\lambda_k)}{f'(\lambda_k)}
\end{equation}
converges to the root $\lambda^*$ of $f$, with a locally quadratic rate.
\end{proposition}
\begin{proof}
First notice that $f''(\lambda)\geq 0$ on the interval $[1/2,1)$. 
Hence $f''$ is also positive on $I=[\lambda^*,\lambda_0]$. 
This ensures that 
\begin{equation}
\lambda_0 \geq \lambda_1 \geq \hdots \geq \lambda^*. 
\end{equation}
We prove this assertion by recurrence. Notice that $\lambda_0\geq \lambda^*$ by Proposition \ref{prop:encadrement}.
Now, assume that $\lambda_k\geq \lambda^*$, then 
\begin{equation}
f(\lambda_k) = f(\lambda^*) + \int_{\lambda^*}^{\lambda_k} f'(t)\,dt \leq f'(\lambda_k) (\lambda_k-\lambda^*).
\end{equation}
Hence, $\lambda_k-\lambda^* \geq \frac{f(\lambda_k)}{f'(\lambda_k)}$ and $\lambda_{k+1}\geq \lambda^*$.
In addition $\frac{f(\lambda_k)}{f'(\lambda_k)}\geq 0$ on $I$, so that $\lambda_{k+1}\geq \lambda_k$.

The sequence $(\lambda_k)_{k\in \N}$ is monotonically decreasing and bounded below, therefore it converges to some value $\lambda'\geq \lambda^*$. Necessarily $\lambda'=\lambda^*$, since for $\lambda'>\lambda^*$, $\frac{f(\lambda')}{f'(\lambda')}>0$.

To prove the locally quadratic convergence rate, we just invoke the celebrated Newton-Kantorovich's theorem \cite{polyak2007newton,ortega1968newton}, that ensures local quadratic convergence if $f''$ is bounded in a neighborhood of the minimizer.
\end{proof}

Finally, let us mention that computing $\lambda_0$ on a computer is a tricky due to underflow problems:
in double precision the command $1+\exp(y_2-y_1)$ will return $1$ for $y_2-y_1<-37\simeq \log(10^{-16})$. This may cause the algorithm to fail since $f$ and its derivatives are undefined at $\lambda=1$. In practice we therefore set $\lambda_0=1/(1+\exp(y_2-y_1))-10^{-16}$. Similarly, by bound \eqref{eq:boundlse}, we get $\lambda^*=1$ up to machine precision whenever $y_2-y_2-a<\log(10^{-16})$. Algorithm \ref{alg:Newton} summarizes all the ideas described in this paragraph.

An attentive reader may have remarked that the convergence of Newton's algorithm depends only on the difference $y(1)-y(2)$ and $a$. A shift of $y(1)$ and $y(2)$ by the same value does not change Newton's iteration. In Fig. \ref{fig:convlse}, we show that the algorithm behaves very well for a wide range of parameters. 
For $y(1)-y(2)$ and $a$ varying  in the interval $[2^{-10},2^{20}]$, the algorithm never requires more than $18$ iterations to reach machine precision and needs $2.8$ iterations in average.

\begin{algorithm}
\caption{An algorithm to compute $\prox_{a\lse}(y_1,y_2)$ with machine precision}
\label{alg:Newton}
\begin{algorithmic}[1]
 \State \textbf{Input:} $(y_1,y_2)\in \R^2, a\in \R_+$.
 \State \textbf{Output:} $(x_1,x_2)=\prox_{a\lse(y_1,y_2)}$.
 \State Set $\epsilon=10^{-16}$.
 \If{$y_1\geq y_2$}
 \If{$y_2-y_1+a<\log(\epsilon)$} 
 \State Set $\lambda=1$.
 \Else 
    \State Set $\lambda=\frac{1}{1+\exp(y_2-y_1)}-\epsilon$.
    \State Define  $d(\lambda)=\frac{y_2-y_1-a+2a\lambda+\log(\lambda/(1-\lambda))}{2a+\frac{1}{\lambda(1-\lambda)}}$.
     \While{$|d(\lambda)|>\epsilon$}
	   \State Set $\lambda = \lambda - d(\lambda)$.
	\EndWhile
 \EndIf
 \State Set $[x_1,x_2]=[y(1)-a\lambda,y(2)-a(1-\lambda)]$.
 \ElsIf{$y_1<y_2$}
 \If{$y_1-y_2+a<\log(\epsilon)$} 
 \State Set $\lambda=1$.
 \Else 
    \State Set $\lambda=\frac{1}{1+\exp(y_1-y_2)}-\epsilon$.
    \State Define $d=\frac{y_1-y_2-a+2a\lambda+\log(\lambda/(1-\lambda))}{2a+\frac{1}{\lambda(1-\lambda)}}$.
     \While{$|d|>\epsilon$}
	   \State Set $\lambda = \lambda - d(\lambda)$.
	\EndWhile
 \EndIf
 \State Set $[x_1,x_2]=[y(1)-a(1-\lambda),y(2)-a\lambda]$.
 \EndIf 
\end{algorithmic}
\end{algorithm}

\begin{figure}
 \centering
 \includegraphics[width=0.45\textwidth]{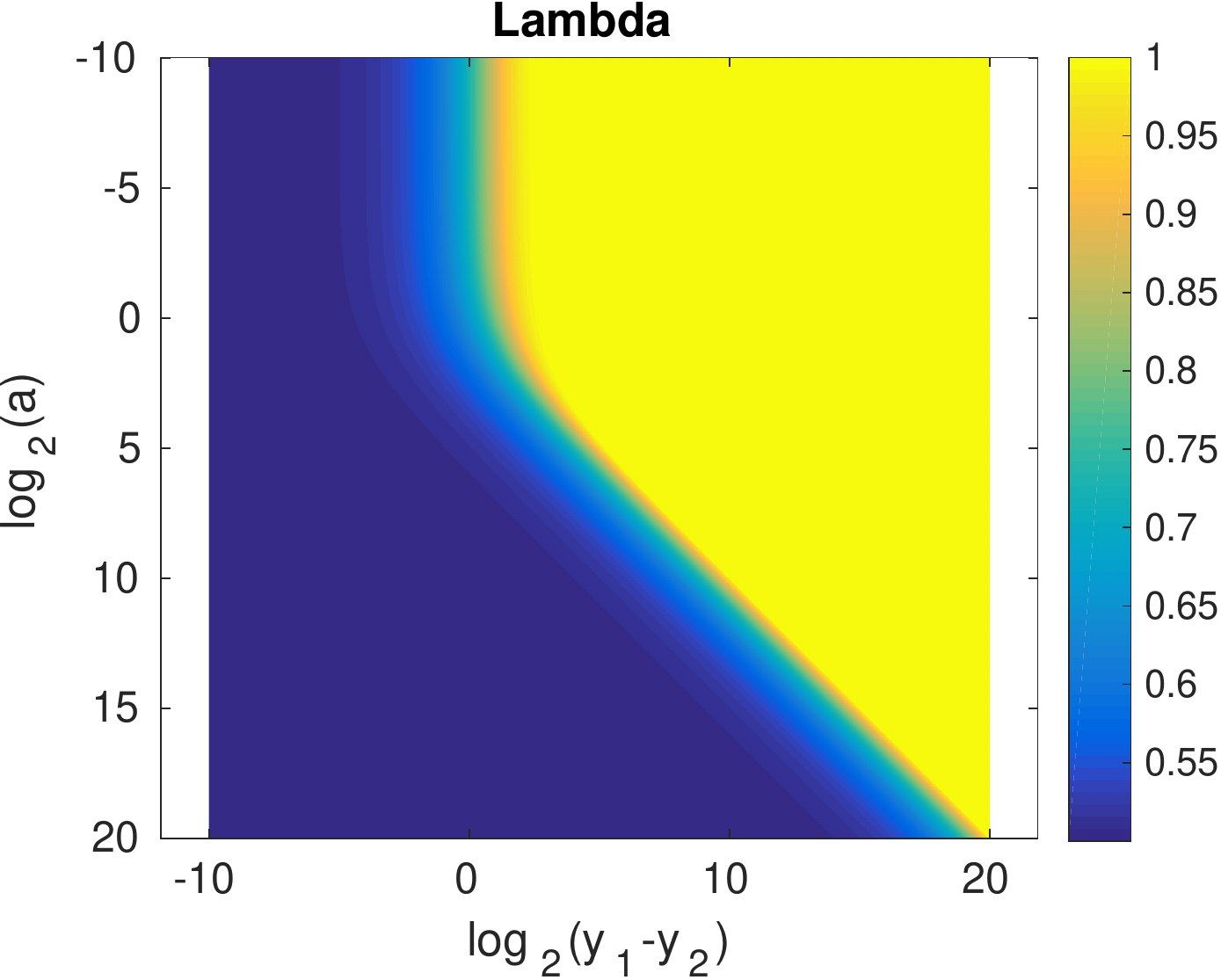} \ \
 \includegraphics[width=0.45\textwidth]{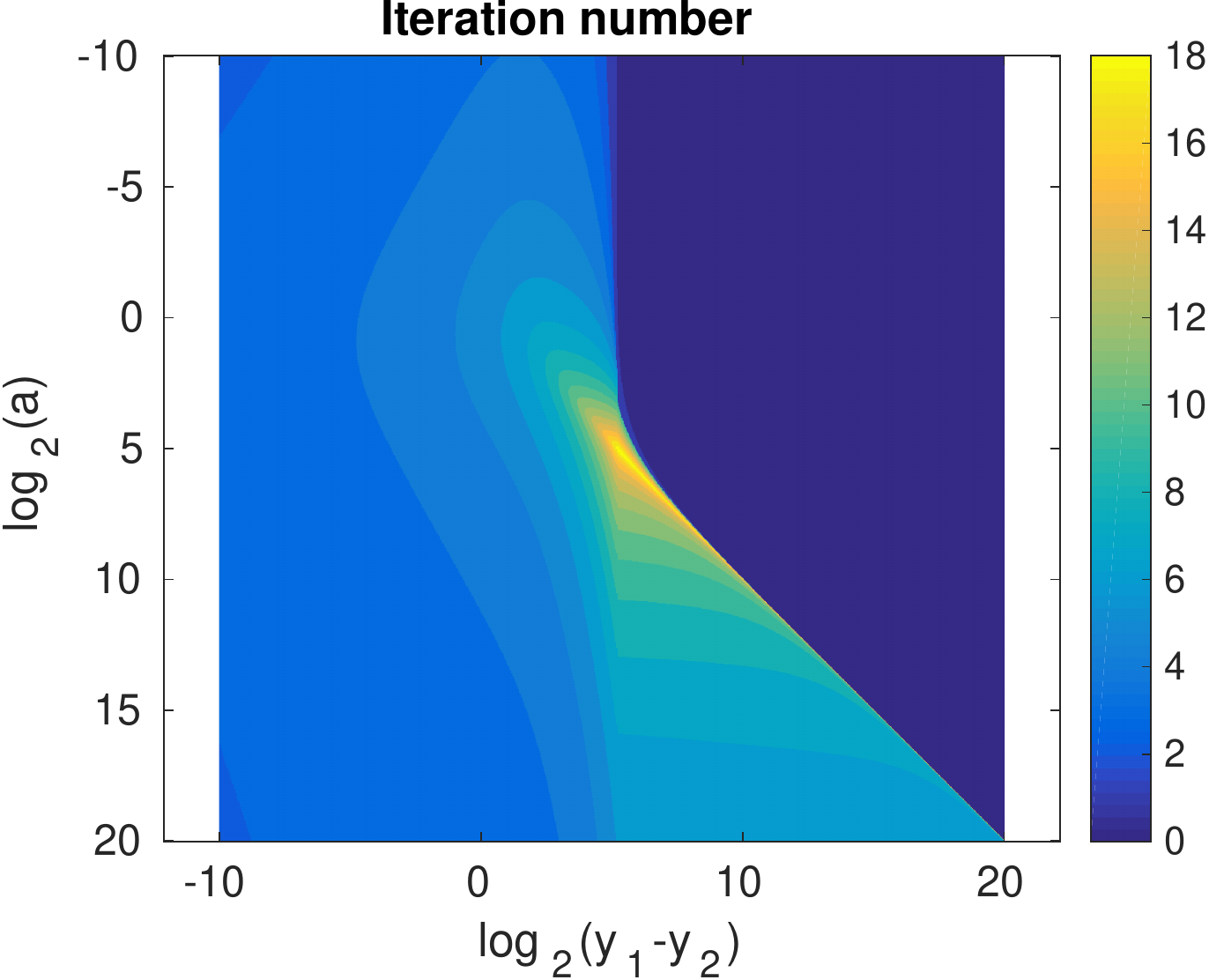}
 \caption{Performance evaluation for Newton's algorithm. Left: found value of $\lambda^*$ depending on $a$ and $y_1-y_2$. Right: number of iterations of Newton's method to reach machine precision.}\label{fig:convlse}
\end{figure}

\subsection{Computing the Lambert W function for large numbers}\label{appendix:lambertw}

In this section, we propose a numerical algorithm to solve the following equation:
\begin{equation}\label{eq:W}
 x\exp(x)=\exp(z).
\end{equation}
The solution is given by $x^*=W(\exp(z))$, where $W$ is the Lambert W function.

When $z$ is sufficiently small, Halley's method proposed in \cite{corless1996lambertw} can be applied. 
However, this method fails whenever $z$ is too large due to numerical instabilities. 
For instance, the command $\log(\exp(z))$ returns $+\infty$ for $z>710$.
In that case, it makes sense to apply a logarithm on each side of equation \eqref{eq:W} and solve the following problem instead:
\begin{equation}
 \log(x) + x = z.
\end{equation}
The two first terms of the asymptotic expansion of $W$ yield $W(\exp(z)) \simeq z - \log(z)$ for large $z$. 
This value is a good starting point for a root finding algorithm. Here, we simply propose to use Newton's algorithm for large $z$. The complete method is detailed in Algorithm \ref{alg:LambertW}. It never takes more than 5 iterations to reach machine precision. The initialization $w = \sqrt{5.43 z + 2} - 1$ is obtained using a Pad\'e approximation and the threshold $0.12$ has been determined experimentally in \cite{corless1996lambertw}.

\begin{algorithm}
\caption{An algorithm to compute the Lambert W function}
\label{alg:LambertW}
\begin{algorithmic}[1]
 \State \textbf{Input:} $z\in \R$.
 \State \textbf{Output:} $W(\exp(z))$.
 \State Set $\epsilon=10^{-16}$.
 \If{$z>0.12$}
 \State Set $w=z-\log(z)$.
 \State Define $d(w)=(\log(w)+w-z)/(1/w+1)$.
 \While{$|d(w)|>\epsilon$}
  \State Set $w = w - d(w)$.
  \EndWhile
 \Else \ \ (algorithm from \cite{corless1996lambertw})
  \State Set $t=\exp(t)$.
  \State Set $w = \sqrt{5.43 t + 2} - 1$.
  \State Define $d(w)=\frac{(w\exp(w)-t)}{\exp(w)w - (w+2)(w\exp(w)-t)/2w}$.
  \While{$|d(w)|>\epsilon$}
    \State Set $w = w - d(w)$.
  \EndWhile
 \EndIf
\end{algorithmic}
\end{algorithm}

\section*{Acknowledgment}

V. Debarnot is supported by Plan CANCER, MIMMOSA project.
The authors wish to thank Emilio Gualda, Jan Huisken, Philipp Keller, Th\'eo Liu, J\"urgen Mayer and Anne Sentenac for interesting discussions and feedbacks on the model.


\bibliographystyle{IEEEtran}
\bibliography{Biblio}

\begin{thebibliography}{10}
\providecommand{\url}[1]{#1}
\csname url@samestyle\endcsname
\providecommand{\newblock}{\relax}
\providecommand{\bibinfo}[2]{#2}
\providecommand{\BIBentrySTDinterwordspacing}{\spaceskip=0pt\relax}
\providecommand{\BIBentryALTinterwordstretchfactor}{4}
\providecommand{\BIBentryALTinterwordspacing}{\spaceskip=\fontdimen2\font plus
\BIBentryALTinterwordstretchfactor\fontdimen3\font minus
  \fontdimen4\font\relax}
\providecommand{\BIBforeignlanguage}[2]{{%
\expandafter\ifx\csname l@#1\endcsname\relax
\typeout{** WARNING: IEEEtran.bst: No hyphenation pattern has been}%
\typeout{** loaded for the language `#1'. Using the pattern for}%
\typeout{** the default language instead.}%
\else
\language=\csname l@#1\endcsname
\fi
#2}}
\providecommand{\BIBdecl}{\relax}
\BIBdecl

\bibitem{natterer1986mathematics}
F.~Natterer, \emph{{The mathematics of computerized tomography}}.\hskip 1em
  plus 0.5em minus 0.4em\relax Siam, 1986, vol.~32.

\bibitem{sharpe2002optical}
J.~Sharpe, U.~Ahlgren, P.~Perry, B.~Hill, A.~Ross, J.~Hecksher-S{\o}rensen,
  R.~Baldock, and D.~Davidson, ``{Optical projection tomography as a tool for
  3D microscopy and gene expression studies},'' \emph{Science}, vol. 296, no.
  5567, pp. 541--545, 2002.

\bibitem{vermeer2014depth}
K.~Vermeer, J.~Mo, J.~Weda, H.~Lemij, and J.~de~Boer, ``Depth-resolved
  model-based reconstruction of attenuation coefficients in optical coherence
  tomography,'' \emph{Biomedical optics express}, vol.~5, no.~1, pp. 322--337,
  2014.

\bibitem{weitkamp2006lidar}
C.~Weitkamp, \emph{{Lidar: range-resolved optical remote sensing of the
  atmosphere}}.\hskip 1em plus 0.5em minus 0.4em\relax Springer Science \&
  Business, 2006, vol. 102.

\bibitem{kunz1987bipath}
G.~J. Kunz, ``{Bipath method as a way to measure the spatial backscatter and
  extinction coefficients with lidar},'' \emph{Applied optics}, vol.~26, no.~5,
  pp. 794--795, 1987.

\bibitem{hughes1988double}
H.~G. Hughes and M.~R. Paulson, ``{Double-ended lidar technique for aerosol
  studies},'' \emph{Applied optics}, vol.~27, no.~11, pp. 2273--2278, 1988.

\bibitem{cuesta2010lidar}
J.~Cuesta and P.~H. Flamant, ``{Lidar beams in opposite directions for quality
  assessment of Cloud-Aerosol Lidar with Orthogonal Polarization spaceborne
  measurements},'' \emph{Applied optics}, vol.~49, no.~12, pp. 2232--2243,
  2010.

\bibitem{mayer2014optispim}
J.~Mayer, A.~Robert-Moreno, R.~Danuser, J.~V. Stein, J.~Sharpe, and J.~Swoger,
  ``{OPTiSPIM: integrating optical projection tomography in light sheet
  microscopy extends specimen characterization to nonfluorescent contrasts},''
  \emph{Optics letters}, vol.~39, no.~4, pp. 1053--1056, 2014.

\bibitem{klett1981stable}
J.~D. Klett, ``{Stable analytical inversion solution for processing lidar
  returns},'' \emph{Applied Optics}, vol.~20, no.~2, pp. 211--220, 1981.

\bibitem{ansmann1990measurement}
A.~Ansmann, M.~Riebesell, and C.~Weitkamp, ``{Measurement of atmospheric
  aerosol extinction profiles with a Raman lidar},'' \emph{Optics letters},
  vol.~15, no.~13, pp. 746--748, 1990.

\bibitem{shcherbakov2007regularized}
V.~Shcherbakov, ``{Regularized algorithm for Raman lidar data processing},''
  \emph{Applied optics}, vol.~46, no.~22, pp. 4879--4889, 2007.

\bibitem{pornsawad2008ill}
P.~Pornsawad, C.~B{\"o}ckmann, C.~Ritter, and M.~Rafler, ``{Ill-posed retrieval
  of aerosol extinction coefficient profiles from Raman lidar data by
  regularization},'' \emph{Applied optics}, vol.~47, no.~10, pp. 1649--1661,
  2008.

\bibitem{garbarino2016expectation}
S.~Garbarino, A.~Sorrentino, A.~M. Massone, A.~Sannino, A.~Boselli, X.~Wang,
  N.~Spinelli, and M.~Piana, ``{Expectation maximization and the retrieval of
  the atmospheric extinction coefficients by inversion of Raman lidar data},''
  \emph{Optics Express}, vol.~24, no.~19, pp. 21\,497--21\,511, 2016.

\bibitem{cremer1974considerations}
C.~Cremer and T.~Cremer, ``{Considerations on a laser-scanning-microscope with
  high resolution and depth of field},'' \emph{Microscopica acta}, pp. 31--44,
  1974.

\bibitem{hell1992properties}
S.~Hell and E.~H. Stelzer, ``{Properties of a 4Pi confocal fluorescence
  microscope},'' \emph{JOSA A}, vol.~9, no.~12, pp. 2159--2166, 1992.

\bibitem{huisken2004optical}
J.~Huisken, J.~Swoger, F.~Del~Bene, J.~Wittbrodt, and E.~H. Stelzer, ``Optical
  sectioning deep inside live embryos by selective plane illumination
  microscopy,'' \emph{Science}, vol. 305, no. 5686, pp. 1007--1009, 2004.

\bibitem{huisken2007even}
J.~Huisken and D.~Y. Stainier, ``{Even fluorescence excitation by
  multidirectional selective plane illumination microscopy (mSPIM)},''
  \emph{Optics letters}, vol.~32, no.~17, pp. 2608--2610, 2007.

\bibitem{krzic2012multiview}
U.~Krzic, S.~Gunther, T.~E. Saunders, S.~J. Streichan, and L.~Hufnagel,
  ``Multiview light-sheet microscope for rapid in toto imaging,'' \emph{Nature
  methods}, vol.~9, no.~7, pp. 730--733, 2012.

\bibitem{tomer2012quantitative}
R.~Tomer, K.~Khairy, F.~Amat, and P.~J. Keller, ``{Quantitative high-speed
  imaging of entire developing embryos with simultaneous multiview light-sheet
  microscopy},'' \emph{Nature methods}, vol.~9, no.~7, pp. 755--763, 2012.

\bibitem{chhetri2015whole}
R.~K. Chhetri, F.~Amat, Y.~Wan, B.~H{\"o}ckendorf, W.~C. Lemon, and P.~J.
  Keller, ``{Whole-animal functional and developmental imaging with isotropic
  spatial resolution},'' \emph{Nature methods}, 2015.

\bibitem{rigaut1991high}
J.~P. Rigaut and J.~Vassy, ``High-resolution three-dimensional images from
  confocal scanning laser microscopy. quantitative study and mathematical
  correction of the effects from bleaching and fluorescence attenuation in
  depth.'' \emph{Analytical and quantitative cytology and histology/the
  International Academy of Cytology [and] American Society of Cytology},
  vol.~13, no.~4, pp. 223--232, 1991.

\bibitem{roerdink1993fft}
J.~Roerdink and M.~Bakker, ``An fft-based method for attenuation correction in
  fluorescence confocal microscopy,'' \emph{Journal of microscopy}, vol. 169,
  no.~1, pp. 3--14, 1993.

\bibitem{can2003attenuation}
A.~Can, O.~Al-Kofahi, S.~Lasek, D.~Szarowski, J.~Turner, and B.~Roysam,
  ``Attenuation correction in confocal laser microscopes: a novel two-view
  approach,'' \emph{Journal of microscopy}, vol. 211, no.~1, pp. 67--79, 2003.

\bibitem{kervrann2004robust}
C.~Kervrann, D.~Legland, and L.~Pardini, ``Robust incremental compensation of
  the light attenuation with depth in 3d fluorescence microscopy,''
  \emph{Journal of Microscopy}, vol. 214, no.~3, pp. 297--314, 2004.

\bibitem{rudin1992nonlinear}
L.~I. Rudin, S.~Osher, and E.~Fatemi, ``Nonlinear total variation based noise
  removal algorithms,'' \emph{Physica D: Nonlinear Phenomena}, vol.~60, no.~1,
  pp. 259--268, 1992.

\bibitem{chambolle2004algorithm}
A.~Chambolle, ``An algorithm for total variation minimization and
  applications,'' \emph{Journal of Mathematical imaging and vision}, vol.~20,
  no. 1-2, pp. 89--97, 2004.

\bibitem{combettes2011proximal}
P.~L. Combettes and J.-C. Pesquet, ``{Proximal splitting methods in signal
  processing},'' in \emph{Fixed-point algorithms for inverse problems in
  science and engineering}.\hskip 1em plus 0.5em minus 0.4em\relax Springer,
  2011, pp. 185--212.

\bibitem{fortin2000augmented}
M.~Fortin and R.~Glowinski, \emph{Augmented Lagrangian methods: applications to
  the numerical solution of boundary-value problems}.\hskip 1em plus 0.5em
  minus 0.4em\relax Elsevier, 2000, vol.~15.

\bibitem{ng2010solving}
M.~K. Ng, P.~Weiss, and X.~Yuan, ``Solving constrained total-variation image
  restoration and reconstruction problems via alternating direction methods,''
  \emph{SIAM journal on Scientific Computing}, vol.~32, no.~5, pp. 2710--2736,
  2010.

\bibitem{corless1996lambertw}
R.~M. Corless, G.~H. Gonnet, D.~E. Hare, D.~J. Jeffrey, and D.~E. Knuth, ``On
  the lambertw function,'' \emph{Advances in Computational mathematics},
  vol.~5, no.~1, pp. 329--359, 1996.

\bibitem{dupe2009proximal}
F.-X. Dup{\'e}, J.~M. Fadili, and J.-L. Starck, ``A proximal iteration for
  deconvolving poisson noisy images using sparse representations,'' \emph{IEEE
  Transactions on Image Processing}, vol.~18, no.~2, pp. 310--321, 2009.

\bibitem{steidl2010removing}
G.~Steidl and T.~Teuber, ``Removing multiplicative noise by douglas-rachford
  splitting methods,'' \emph{Journal of Mathematical Imaging and Vision},
  vol.~36, no.~2, pp. 168--184, 2010.

\bibitem{nikolova2007model}
M.~Nikolova, ``Model distortions in bayesian map reconstruction,''
  \emph{Inverse Problems and Imaging}, vol.~1, no.~2, p. 399, 2007.

\bibitem{gribonval2011should}
R.~Gribonval, ``Should penalized least squares regression be interpreted as
  maximum a posteriori estimation?'' \emph{IEEE Transactions on Signal
  Processing}, vol.~59, no.~5, pp. 2405--2410, 2011.

\bibitem{starck2013sparsity}
J.-L. Starck, D.~Donoho, M.~Fadili, and A.~Rassat, ``Sparsity and the bayesian
  perspective,'' \emph{Astronomy \& Astrophysics}, vol. 552, p. A133, 2013.

\bibitem{nlm}
A.~Buades, B.~Coll, and J.~M. Morel, ``A review of image denoising algorithms,
  with a new one,'' \emph{Multisc. Model. Simulat.}, vol.~4, no.~2, pp.
  490--530, 2005.

\bibitem{hiriart2006note}
J.-B. Hiriart-Urruty, ``A note on the legendre-fenchel transform of convex
  composite functions,'' in \emph{Nonsmooth Mechanics and Analysis}.\hskip 1em
  plus 0.5em minus 0.4em\relax Springer, 2006, pp. 35--46.

\bibitem{polyak2007newton}
B.~T. Polyak, ``Newton’s method and its use in optimization,'' \emph{European
  Journal of Operational Research}, vol. 181, no.~3, pp. 1086--1096, 2007.

\bibitem{ortega1968newton}
J.~M. Ortega, ``The newton-kantorovich theorem,'' \emph{The American
  Mathematical Monthly}, vol.~75, no.~6, pp. 658--660, 1968.

\end{thebibliography}

\end{document}